\documentclass[sn-mathphys,Numbered]{sn-jnl}% Math and Physical Sciences Reference Style
\usepackage{graphicx}%
\usepackage{multirow}%
\usepackage{amsmath,amssymb,amsfonts}%
\numberwithin{equation}{section}%equation labels
\usepackage{amsthm}%
\usepackage{mathrsfs}%
\usepackage[title]{appendix}%
\usepackage{xcolor}%
\usepackage{textcomp}%
\usepackage{manyfoot}%
\usepackage{booktabs}%
\usepackage{multirow}
\usepackage[linesnumbered,ruled]{algorithm2e}
\usepackage{algpseudocode}%
\usepackage{listings}%
\usepackage{enumerate}
\usepackage{caption}
\usepackage{subcaption}
\usepackage{verbatim}
\usepackage{hyperref}
\graphicspath{{img/}}

\usepackage{color}

\usepackage{fix-cm}

\theoremstyle{thmstyleone}%
\newtheorem{thm}{Theorem}[section]
\newtheorem{lem}{Lemma}[section]
\newtheorem{assu}{Assumption}

\theoremstyle{thmstyletwo}%
\newtheorem{remark}{Remark}[section]

\theoremstyle{thmstylethree}%
\newtheorem{defi}{Definition}[section]

\theoremstyle{thmstylefour}%

\begin{document}

\title[DDRSM for weakly convex problems]{A distributed Douglas-Rachford splitting method for solving linear constrained multi-block weakly convex problems\footnotemark[0]}

\author[1]{\fnm{Leyu} \sur{Hu}}\email{huleyu@buaa.edu.cn}

\author[1]{\fnm{Jiaxin} \sur{Xie}}\email{xiejx@buaa.edu.cn}

\author[2]{\fnm{Xingju} \sur{Cai}} \email{caixingju@njnu.edu.cn}

\author*[1]{\fnm{Deren} \sur{Han}}\email{handr@buaa.edu.cn}

\affil[1]{\orgdiv{LMIB of the Ministry of Education}, \orgname{School of Mathematical Sciences, Beihang University}, \orgaddress{\city{Beijing}, \postcode{100191}, \country{China}}}

\affil[2]{\orgdiv{School of Mathematical Sciences}, \orgname{Nanjing Normal University}, \orgaddress{\city{Nanjing}, \postcode{210023}, \country{China}}}

\date{}
\abstract{
In recent years, a distributed Douglas-Rachford splitting method (DDRSM) has been proposed to tackle multi-block separable convex optimization problems. This algorithm offers relatively easier subproblems and greater efficiency for large-scale problems compared to various augmented-Lagrangian-based parallel algorithms. Building upon this, we explore the extension of DDRSM to weakly convex cases. By assuming weak convexity of the objective function and introducing an error bound assumption, we demonstrate the linear convergence rate of DDRSM. Some promising numerical experiments involving compressed sensing and robust alignment of structures across images (RASL) show that DDRSM has advantages over augmented-Lagrangian-based algorithms, even in weakly convex scenarios.
}

\keywords{multi-block problems, weakly convex, parallel algorithm, distributed Douglas-Rachford splitting method, error bound, linear convergence rate}
\msc{90C26, 90C30, 65K10, 94A08}

\maketitle

\renewcommand{\thefootnote}{}
\footnotetext[0]{The study was supported by the Ministry of Science and Technology of China (No. 2021YFA1003600), and by National Natural Science Foundation of China (12126608, 12126603, 12471290)}

\section{Introduction} \label{sec1}

We consider the following multi-block separable optimization problem with linear constraints. The objective function is the sum of $m$ individual functions but without coupled variables:
\begin{equation}\label{DRSO}
	\min_{x_1, x_2, \cdots, x_m} \left\{\sum_{i=1}^m f_i\left(x_i\right) \mid \sum_{i=1}^m A_i x_i=b, \quad x_i \in \mathcal{X}_i, \quad i=1,2, \cdots, m\right\},
\end{equation}
where $f_i:\mathbb{R}^{n_i} \rightarrow \mathbb{R}\cup\{+\infty\},\; i=1,2,\dots,m$ 
are weakly convex functions;
$A_i\in\mathbb{R}^{l\times n_i}$ are given matrices; $b\in \mathbb{R}^l$ is a given vector; $\mathcal{X}_i\subseteq \mathbb{R}^{n_i}$ are closed nonempty convex sets. 
We further assume that the solution set of \eqref{DRSO} is nonempty, and there exists a solution point satisfying the KKT conditions.
We use notations $A:=(A_1, A_2, \dots, A_m)\in \mathbb{R}^{l\times n}$, $x:=(x_1^\top,x_2^\top,\dots,x_m^\top)^\top\in\mathbb{R}^n$, $\mathcal{X}:=\mathcal{X}_1 \times \mathcal{X}_2 \times \cdots \times \mathcal{X}_m\subseteq \mathbb{R}^n$, and $f(x):=\sum_{i=1}^m f_i\left(x_i\right)$, where $n=\sum_{i=1}^{m} n_i$.

The form above is a unified framework for various models. By using suitable smoothing techniques, several nonconvex norm regularization problems can be expressed  as \eqref{DRSO}. Thus, the structure encompasses a wide range of applications such as compressed sensing \cite{zhao2016nonconvex,wei2016nonlocal}, image restoration \cite{zhang2017nonconvex,chen2013bregman}, image retinex \cite{xu2020star,liu2017novel,wang2020image,liu2018retinex,hu2020parallel}, sparse optimization \cite{zeng2014l_}, and so on.

Optimization problems like \eqref{DRSO} often involve high dimensions in practical applications.
To exploit the inherent multi-block structure and solve it in parallel, it is necessary to decompose the problem into multiple subproblems for more efficient computation.
Researchers have extensively investigated algorithms for the convex case of problem \eqref{DRSO}, which involves closed convex functions $f_i$. Among these algorithms, augmented-Lagrangian-based splitting methods stand out as a significant class.
The iterative formulation of the augmented Lagrangian method (ALM) \cite{hestenes1969multiplier} is as follows:
\begin{equation}
	\left\{\begin{array}{l}
		\left(x_1^{k+1}, \cdots, x_m^{k+1}\right)=\underset{x \in \mathcal{X}}{\operatorname{argmin} }\;\mathcal{L}_\beta(x_1,\cdots,x_m,\lambda^k),\\
		\lambda^{k+1}=\lambda^k-\beta\left(\sum_{i=1}^m A_i x_i^{k+1}-b\right),
	\end{array}\right.\label{eq:ALM}
\end{equation}
where $\beta>0$ is the penalty parameter, and the augmented Lagrangian function is defined as
\begin{equation}\label{eq:lagrangian}
	\mathcal{L}_\beta(x_1,\cdots,x_m,\lambda):=\sum_{i=1}^m f_i\left(x_i\right)-\left\langle\lambda, \sum_{i=1}^m A_i x_i-b\right\rangle+\frac{\beta}{2}\left\|\sum_{i=1}^m A_i x_i-b\right\|^2.
\end{equation}

The subproblem in ALM \eqref{eq:ALM} can be challenging to solve because the variables $x_i, i=1,\dots,m$ are coupled through the quadratic term $\frac{\beta}{2}\left\|\sum_{i=1}^m A_i x_i-b\right\|^2$.
To exploit the multi-block structure of problem \eqref{DRSO}, we adopt the alternating direction method of multipliers (ADMM) as a variant of ALM. The subproblems of multi-block ADMM are solved in a Gauss-Seidel fashion.
For convex cases, when $m\geq 3$, Chen \emph{et al.} \cite{chen2016direct} show that ADMM does not necessarily converge without additional assumptions or modifications.
However, for $m=3$, with at least one block being strongly convex, Cai \emph{et al.} \cite{cai2017convergence} proved its convergence.
For nonconvex cases, the convergence of the two-block ADMM has been established in \cite{wang2014convergence, li2015global, yang2017alternating} under the Kurdyka-\L ojasiewicz (KL) assumption.
Extending these results to the multi-block case, Guo \emph{et al.} \cite{guo2017convergence} demonstrated the convergence of multi-block ADMM for nonconvex problems.
Some of these studies also established convergence rates based on the KL property.
Additionally, Jia \emph{et al.} \cite{jia2021local} established a local linear convergence rate of ADMM for nonconvex problems by leveraging an error bound condition.
This error bound condition is formulated as follows: for $w = (x^\top,\lambda^\top)^\top$ and a solution set $W^*$,
\begin{equation*}	
	\operatorname{dist}(w,W^*) \leq \tau \|E_\beta(x,\lambda)\|,
\end{equation*}
where $E_\beta$ represents an error function.
In our work, we introduce this error bound assumption, and discuss its justification and applicability in Section \ref{sec:converAnalysis}. Further foundational details on this error bound condition can be found in the works of Luo and Tseng \cite{luo1992linear, luo1993error}.

To further improve the efficiency of solving subproblems in \eqref{eq:ALM}, an alternative approach is to decompose them into $m$ simpler and smaller subproblems using Jacobian decomposition techniques for the quadratic term.
This Jacobian-type approach allows the subproblems to be solved simultaneously.
However, to guarantee the convergence of the algorithm, it requires an additional correction step or modifications to the subproblems, as demonstrated in previous works \cite{han2014proximal, han2014augmented, he2009proximal, he2015full, deng2017parallel, jiang2022efficient}.
Convergence analyses of Jacobian-type splitting methods for nonconvex problems have been conducted under various assumptions.  For instance, Chatzipanagiotis and Zavlanos \cite{chatzipanagiotis2017convergence} established convergence by relying primarily on strong second-order sufficient conditions at the solution point $(x^*,\lambda^*)$, whereas Yashtini \cite{yashtini2021multi} relied on KL assumptions.

While the aforementioned ALM-based algorithms leverage the advantageous separable structure of multi-block problems, they still have areas that can be improved. 
Firstly, solving subproblems can be challenging due to constraints such as $x_i \in \mathcal{X}_i$ and the inclusion of square terms with $Ax_i$. 
These factors often lead to subproblems that do not have closed-form solutions. 
Secondly, reducing the number of correction variables in algorithms that involve correction steps may lead to more favorable numerical results, as observed by \cite{tao2011recovering,he2012alternating}.

In addition to ALM-based algorithms, operator splitting methods are commonly employed for solving certain special cases of problem \eqref{DRSO}. Specifically, when we focus on problems where $m = 2$, consider only the linear constraints (without the abstract constraints $x_i \in \mathcal{X}_i$), and particularly when these linear constraints are restricted to $x_1 = x_2$, many operator splitting algorithms can be effectively applied. In such cases, we can reformulate the first-order optimality conditions of the two objective functions into a generalized inclusion problem involving their gradient operators, as follows:
\begin{equation}\label{eq:DR-2block}
	0 \in \partial f_1 + \partial f_2.	
\end{equation}
For problems of this form, methods like the Douglas-Rachford splitting method (DRSM) are commonly used.
In this context, Li and Pong \cite{li2016douglas} demonstrated convergence for nonconvex problems under the assumption that one of the objective function gradients is Lipschitz continuous, leveraging the KL property. 
Guo \emph{et al.} \cite{guo2017convergenceDRSM, guo2018note} provided convergence proofs for cases where the objective functions are strongly convex and weakly convex. 
Building upon this, Li and Wu \cite{li2019convergence} extended the results in Li and Pong \cite{li2016douglas} to Peaceman-Rachford splitting method (PRSM), and Themelis and Patrinos \cite{themelis2020douglas} established tight convergence results for DRSM and PRSM and extended this scheme for ADMM.

Addressing the limitations of the classic Douglas-Rachford (DR) approach in the context of problems with $m \geq 3$, where \eqref{eq:DR-2block} becomes a multi-block inclusion problem:
\begin{equation}\label{eq:DR-multiblock}
0\in \partial f_1 + \partial f_2 + \cdots + \partial f_m,
\end{equation}
and considering general constraints, the classic DRSM becomes inapplicable. 

To tackle issues of ALM-based algorithms and the limitations of the classic DRSM, He and Han \cite{he2016distributed} reformulated the KKT system of problem \eqref{DRSO} into a generalized inclusion problem with two operator blocks:
\begin{equation}\label{eq:inclusion}
0 \in F(u) + \mathcal{N}_{\Omega}(u),
\end{equation}
where $u$ is a combined vector of variables and multipliers,
$F(u)$ represents the gradients of the Lagrangian functions with respect to the variables and multipliers,
and $\mathcal{N}_{\Omega}(u)$ is the normal cone to the constraint set. 
Instead of adopting \eqref{eq:DR-multiblock},
He and Han applied DRSM to this inclusion problem \eqref{eq:inclusion}, thereby proposing the Distributed Douglas-Rachford splitting method (DDRSM). 
The specific correspondence between this problem and problem \eqref{DRSO} will be detailed in Section \ref{sec:equivalent}. 
This innovative approach facilitates the use of DRSM to solve problems with $m \geq 3$ and general linear constraints. 
The DDRSM scheme can be succinctly described as follows:
\begin{equation}\label{eq:DDRSM scheme}
	\left\{\begin{array}{l}\text{Generate a predictor } \bar{\lambda}^k \text{ with given }\left(x_1^k,\cdots,x_m^k, \lambda^k\right), \\ x_i^{k+1}=\underset{x_i \in \mathbb{R}^{n_i}}{\operatorname{argmin} }\left\{f_i\left(x_i\right)+\frac{1}{2 \beta}\left\|x_i-\left(x_i^k+\beta \xi_i^k-\rho \alpha_k e_{x_i}^k\left(x_i^k, \bar{\lambda}^k\right)\right)\right\|^2\right\},\;i=1, \cdots, m, \\ \text{Update } \lambda^{k+1} \text{ in an appropriate scheme.} \end{array}\right.
\end{equation}
Here, $\xi_i^k \in \partial f_i(x_i^k)$ denotes a subgradient of the function $f_i$ at $x_i^k$, and $e_{x_i}^k(x_i^k, \bar{\lambda}^k)$ refers to the error functions defined later in \eqref{eq:Ebeta}. Compared to previous ALM-based algorithms, DDRSM offers several advantages: firstly, it provides easily solvable subproblems that can be addressed simultaneously, thereby reducing computational complexity and making it suitable for large-scale problems; secondly, it introduces an additional update of $\lambda$ to ensure convergence.

However, applying DRSM to the KKT equations of nonconvex problems remains challenging for convergence analysis due to the difficulty of constructing a measure function that strictly decreases throughout the algorithm.
This makes it hard to analyze the convergence by leveraging the KL property, since the sufficient decrease of the measure function is not guaranteed.
To address this issue, we introduce a local error bound assumption in this article to prove the convergence of DDRSM for weakly convex problems.
We have modified the parameters of DDRSM to tackle weakly convex problems and also prove a linear convergence rate under these assumptions. 
Through extensive numerical experiments on compressed sensing and robust alignment via sparse and low-rank decomposition, we provide compelling evidence of the effectiveness and versatility of our algorithm.

The rest of this paper is organized as follows:
Section 2 introduces the fundamental concepts, notations, and relevant properties that will be used throughout the paper.
Section 3 introduces an adapted version of DDRSM for weakly convex problems. We concretely specify the selection of parameters.
The linear convergence of the discussed algorithm is proven in Section 4 under the assumption of an error bound condition and weak convexity.
Section 5 demonstrates the effectiveness of our method through numerical experiments on the Lasso model for compressed sensing of signal and the tensor robust principal component analysis model for image alignment.
The paper concludes with some final remarks in Section 6.

\section{Preliminaries}

In this section, we introduce some preliminary concepts and significant results that will be utilized throughout the subsequent parts of our article. For a matrix $A$ and a vector $x$, we use $A^\top$ and $x^\top$ to represent their transposes. For any two vectors $x, y \in \mathbb{R}^n$, $\left\langle x, y \right\rangle$ denotes their standard inner product. We adopt the notation \( \|\cdot\| \) to denote the Euclidean norm for vectors and the 2-norm (spectral norm) for matrices, depending on the context. The $\ell_p$ norm of a vector $x$ is defined as $\|x\|_p := \left(\sum_{i=1}^n |x_i|^p\right)^{1/p}$. The term $\ell_p$ norm is applicable when $p \geq 1$. For values in the $0 < p < 1$ range, the $\ell_p$ norm is identified as a nonconvex quasi-norm. For a matrix $X \in \mathbb{R}^{m \times n}$, the nuclear norm $\|X\|_*$ is defined as the sum of its singular values, which can be expressed as $\|X\|_* = \sum_{i=1}^{\min(m,n)} \sigma_i$, where $\sigma_i$ represents the $i$-th singular value of $X$.

\begin{defi}
	Let $ S $ be a subset of $ \mathbb{R}^n $.
	
	The \emph{affine hull} of $ S $, denoted $ \operatorname{aff}(S) $, is the smallest affine subspace containing $ S $. It consists of all affine combinations of points in $ S $ and is defined as:
	\[
	\operatorname{aff}(S) = \left\{ \sum_{i=1}^{k} \alpha_i x_i \,\bigg|\, x_i \in S,\ \sum_{i=1}^{k} \alpha_i = 1,\ \alpha_i \in \mathbb{R},\ k \in \mathbb{N} \right\}.
	\]
	
	The \emph{relative interior} of $ S $, denoted $ \operatorname{ri}(S) $, is the interior of $ S $ relative to its affine hull. It is defined as:
	\[
	\operatorname{ri}(S) = \left\{ x \in S \,\bigg|\, \exists\, \epsilon > 0 \text{ such that } B(x, \epsilon) \cap \operatorname{aff}(S) \subseteq S \right\},
	\]
	where $ B(x, \epsilon) $ denotes the open ball of radius $ \epsilon $ centered at $ x $.
	
	\end{defi}

For a function $f: \mathbb{R}^n \rightarrow \mathbb{R} \cup \{+\infty\}$, $\operatorname{dom} f$ is the \textit{effective domain} of $f$, which is defined as $\operatorname{dom} f = \{ x \in \mathbb{R}^n \mid f(x) < +\infty \}$. We use $\nabla f$ to denote its gradient and $\partial f$ to denote its \textit{subdifferential}, which is defined as follows.

\begin{defi}
	Let $f: \mathbb{R}^n \rightarrow \mathbb{R} \cup \{+\infty\}$ be a lower semicontinuous function.
	\begin{enumerate}[(i)]
		\item The \textit{Fr\'echet subdifferential}, or \textit{regular subdifferential}, of $f$ at $x \in \mathbb{R}^n$, denoted as $\hat{\partial} f(x)$, is the set of vectors $x^* \in \mathbb{R}^n$ that satisfy the following condition:
		\[		
		\liminf _{y \neq x,\, y \rightarrow x} \frac{f(y)-f(x)-\left\langle x^*, y-x\right\rangle}{\|y-x\|} \geqslant 0.		
		\]
		
		\item The \textit{limiting subdifferential}, or simply the subdifferential, of $f$ at $x \in \mathbb{R}^n$, written as $\partial f(x)$, is defined as		
		$$		
		\partial f(x):=\left\{ x^* \in \mathbb{R}^n \mid \exists\, x_n \rightarrow x,\ f\left(x_n\right) \rightarrow f(x),\ x_n^* \in \hat{\partial} f\left(x_n\right),\ x_n^* \rightarrow x^* \right\}.		
		$$		
	\end{enumerate}	
\end{defi}

We introduce two important operators that will be frequently used in our analysis. First, the \textit{proximal operator}, denoted as $\operatorname{prox}_{\beta f}(x)$, is defined for a function $f: \mathbb{R}^n \rightarrow \mathbb{R} \cup \{+\infty\}$ as follows:
\begin{equation}\label{eq:proximal}
	\operatorname{prox}_{\beta f}(x) = \underset{y \in \mathbb{R}^n}{\operatorname{argmin}} \left( f(y) + \frac{1}{2 \beta}\|y - x\|^2 \right),
\end{equation}
where $\beta$ is a positive parameter. We say that a function $f$ is \textit{prox-bounded} if $f + \frac{1}{2\beta}\|\cdot\|^2$ is lower bounded for some $\beta > 0$. The function $M_{\beta} f(x)$, defined as
\[
M_{\beta} f(x) = \underset{y \in \mathbb{R}^n}{\inf} \left( f(y) + \frac{1}{2 \beta}\|y - x\|^2 \right),
\]
is called the \textit{Moreau envelope} of $f$.

Second, the \textit{projection operator} onto a closed convex set $C \subseteq \mathbb{R}^n$, denoted as $P_C(x)$, is defined as follows:
\[
P_C(x) = \underset{y \in C}{\operatorname{argmin}} \|y - x\|.
\]
For such $C$, we denote $\mathcal{N}_{C}(u)$ as the \textit{normal cone} of $C$ at $u$, which is defined as
\[
\mathcal{N}_{C}(u) = \left\{ w \in \mathbb{R}^n \mid \left\langle w, v - u \right\rangle \leq 0,\ \forall v \in C \right\}.
\]

Next, we define the \textit{Lipschitz continuity} of the gradient as follows.

\begin{defi}
	A function $f: \mathbb{R}^n \rightarrow \mathbb{R}$ is said to have a Lipschitz continuous gradient if there exists a constant $L > 0$ such that for any $x, y \in \mathbb{R}^n$, the following inequality holds:
	\[
	\|\nabla f(x) - \nabla f(y)\| \leq L \|x - y\|.	
	\]
\end{defi}

Next, we define \textit{weakly convex} functions as follows.

\begin{defi}
	A function $f:\mathcal{X} \subseteq \mathbb{R}^n \rightarrow \mathbb{R}$ is called $c$-weakly convex if there exists a constant $c \geq 0$ such that for all $x, y \in \mathcal{X}$ and any $\xi \in \partial f(x)$, the following inequality holds:
	\[
	f(y) \geq f(x) + \langle \xi, y - x \rangle - \frac{c}{2} \|y - x\|^2.
	\]
\end{defi}

If a function $f$ is $c$-weakly convex, then, similar to strongly convex functions, it can be shown that its subgradient operator $\partial f$ is \textit{$c$-weakly monotone}. This property is defined similarly to strong monotonicity but with a negative constant $c$ in the inequality. The relationship between strong convexity and strong monotonicity is well-known in the literature (see, e.g., \cite[Exercise~12.59]{rockafellar2009variational}).

\begin{lem}
	If the function $f:\mathcal{X} \subseteq \mathbb{R}^n \rightarrow \mathbb{R}$ is $c$-weakly convex for some $c \geq 0$, then its subgradient operator $\partial f$ is $c$-weakly monotone. This means that for all $x, y \in \mathcal{X}$ and $\xi \in \partial f(x),\ \zeta \in \partial f(y)$,
	\begin{equation}\label{eq:weakconvex}
		\left\langle x - y,\, \xi - \zeta \right\rangle \geq -c \|x - y\|^2.
	\end{equation}
\end{lem}
\begin{proof}
	Since $f$ is weakly convex, $f + \frac{c}{2}\|\cdot\|^2$ is convex. Then, by \cite[Theorem~12.17]{rockafellar2009variational}, we have that $\partial \left( f + \frac{c}{2}\|\cdot\|^2 \right)$ is monotone, which implies \eqref{eq:weakconvex}.
\end{proof}

\subsection{Equivalent formulations of problem \eqref{DRSO} and introduction of DDRSM}
\label{sec:equivalent}

In this subsection, we introduce several equivalent formulations of problem \eqref{DRSO}. This exploration serves to establish a clear connection between the Douglas-Rachford splitting method (DRSM) and problem \eqref{DRSO}, facilitating the introduction of the DDRSM scheme \eqref{eq:DDRSM scheme}. Additionally, we present a lemma that characterizes the equivalence between the solution set and the zero points of the error functions $e_{x_i}^k(x_i^k, \bar{\lambda}^k),\ i = 1, \dots, m$ in \eqref{eq:DDRSM scheme}. This lemma will be crucial in the subsequent proofs.

First, we consider the optimality conditions of problem \eqref{DRSO} under certain constraint qualifications. For example, the Slater's condition for \eqref{DRSO} is
\begin{equation}\label{eq:Slater}
	\exists x^\prime \in \operatorname{ri}\left(\operatorname{dom}f\right)\cap\left\{x \in \mathcal{X} \mid Ax = b\right\},
\end{equation}
with $f$, $A$, $b$, and $\mathcal{X}$ as defined in Section~\ref{sec1}.
By \cite[Theorem~5.1.6]{makela1992nonsmooth}, we know that if certain constraint qualifications, such as Slater's condition, hold, then there exist $x^* \in \mathcal{X},\ \lambda^* \in \mathbb{R}^l$ such that
\begin{equation}\label{eq:KKT}
	\left\{\begin{array}{l}
		0 \in \partial f_i(x_i^*) + A_i^\top \lambda^*, \quad x_i^* \in \mathcal{X}_i,\ i = 1, 2, \dots, m,\\
		\sum_{i=1}^m A_i x_i^* = b.
	\end{array}\right.
\end{equation}
Here, $(x^*, \lambda^*)$ is called a KKT point of problem \eqref{DRSO}. We can then derive an equivalent generalized inclusion form of \eqref{eq:KKT} as
\begin{equation}\label{DRSO3}
    0 \in F(u) + \mathcal{N}_{\Omega}(u),
\end{equation}
where
\[
F(u) = \left(\begin{array}{c}
    \xi - A^{\top} \lambda \\
    A x - b
\end{array}\right), \quad u = \left(\begin{array}{c} x \\ \lambda \end{array}\right), \quad \xi \in \partial f(x),
\]
and $\mathcal{N}_{\Omega}(u)$ is the normal cone of $\Omega := \mathcal{X} \times \mathbb{R}^l$.
By the definition of the normal cone, we can rewrite \eqref{eq:KKT} and \eqref{DRSO3} as the following variational inequality:  
find a vector $u^* = \left( \begin{array}{c} x^* \\ \lambda^* \end{array} \right)$, where $x^* \in \mathcal{X}$, $\lambda^* \in \mathbb{R}^l$, and $\xi^* \in \partial f(x^*)$, such that
\begin{equation}\label{DRSOVI}
	\left\{ \begin{array}{l}
		\left\langle x - x^*,\, \xi^* - A^\top \lambda^* \right\rangle \geq 0, \quad \forall x \in \mathcal{X}, \\[1ex]
		\left\langle \lambda - \lambda^*,\, A x^* - b \right\rangle \geq 0, \quad \forall \lambda \in \mathbb{R}^l.
	\end{array} \right.
\end{equation}
As long as the solution set of \eqref{DRSO} is nonempty and a KKT point exists, any solution of \eqref{DRSO} is also a solution of \eqref{DRSO3} or \eqref{DRSOVI}.

For the generalized inclusion problem \eqref{DRSO3}, He and Han \cite{he2016distributed} applied a form of DRSM to it with extrapolation:
\begin{equation}\label{eq:DRSM}
    u^{k+1} + \beta F(u^{k+1}) = u^k + \beta F(u^k) - \rho E_\beta(u^k),
\end{equation}
where $\beta > 0$, and $E_\beta(u^{k})$ is referred to as the \textit{natural map} \cite[p.~83]{facchinei2003finite}, which serves as an error function for the variational inequality \eqref{DRSOVI}. The natural map is defined as
\begin{equation}\label{eq:Ebeta}
	\begin{aligned}
		E_\beta(u) &= E_\beta(x, \lambda) = u - P_{\Omega}(u - \beta F(u)) \\
		&= \left( \begin{array}{c} e_{x}(x, \lambda) \\ e_\lambda(x) \end{array} \right) = \left( \begin{array}{c}
			x - P_{\mathcal{X}}\left( x - \beta ( \xi - A^{\top} \lambda ) \right) \\[1ex]
			\beta (A x - b)
		\end{array} \right).
	\end{aligned}
\end{equation}
According to \cite[Proposition~1.5.8]{facchinei2003finite} and \cite[Lemma~2.1]{he2016distributed}, we know that $u^* = \left( \begin{array}{c} x^* \\ \lambda^* \end{array} \right)$ is a solution of \eqref{DRSOVI} if and only if $E_\beta(x^*, \lambda^*) = 0$.

Based on \eqref{eq:DRSM}, the introduction of a \textit{prediction-correction} scheme, as proposed by Chen \cite{chen1994proximal}, leads to a modification of the natural map, resulting in the error function of DDRSM.
We introduce a predictor for $\lambda$ as
\begin{equation}\label{eq:defbarla}
	\bar{\lambda} = \lambda - e_\lambda(x).
\end{equation}
Since $\bar{\lambda}^*=\lambda^*-e_\lambda(x^*)=\lambda^*$ when $x^*$ is the solution, the following lemma shows that $E_\beta(x,\bar{\lambda})$ has the same zero points as $E_\beta(x,\lambda)$.
\begin{lem}
	The vector $\left(\begin{array}{c}x^*\\ \lambda^*\end{array}\right) \in \Omega$ is a solution of VI \eqref{DRSOVI} if and only if $\left\|E_\beta(x^*,\bar{\lambda}^*)\right\|= 0$ for any $\beta>0$, where $E_\beta$ is defined in \eqref{eq:Ebeta} and $\bar{\lambda}^*=\lambda^*-e_\lambda(x^*)$. \label{lem:naturalmap}
\end{lem}
\begin{proof}
	According to \cite[Lemma~2.1]{he2016distributed}, $\left(\begin{array}{c}x^*\\ \lambda^*\end{array}\right) \in \Omega$ is a solution of VI \eqref{DRSOVI} if and only if ${\|E_\beta(x^*,\lambda^*)\|=0}$.
	So when $\left(\begin{array}{c}x^*\\ \lambda^*\end{array}\right) \in \Omega$ is a solution, we have $e_\lambda(x^*)=0$. Therefore, $\bar{\lambda}^*=\lambda^*$ due to \eqref{eq:defbarla}, and we have $\|E_\beta(x^*,\bar{\lambda}^*)\| = \|E_\beta(x^*,\lambda^*)\| = 0$.
	
	On the other hand, if $\|E_\beta(x^*,\bar{\lambda}^*)\|=0$, we also have $e_\lambda(x^*)=0$, since it is a component of $E$. 
	Then, $\|E_\beta(x^*,\lambda^*)\| = 0$, which concludes the proof.
\end{proof}

\section{Distributed Douglas-Rachford splitting method for weakly convex problems}

In this section, we elaborate on the framework \eqref{DRSO}, propose a distributed Douglas-Rachford splitting method for solving weakly convex problems, and present two compact formulations of the algorithm. 
Before presenting the algorithm, we introduce some simple assumptions to ensure that the algorithm is well-defined.

\begin{assu}\label{assu:f}
	The objective functions $ f_i $ are weakly convex functions and the proximal operator of each function $f_i$ defined in \eqref{eq:proximal}	is computationally efficient to evaluate.
\end{assu}
\begin{remark}\label{remark:f-prox}
	It is easy to verify that $f_i$ satisfying Assumption \ref{assu:f} is prox-bounded, and there exists $\beta>0$ such that $f_i + \frac{1}{2\beta}\|\cdot\|^2$ is strongly convex. Thus, $\operatorname{prox}_{\beta f_i}(\cdot)$ is single-valued.
\end{remark}
\begin{assu}\label{assu:sets}
	The abstract constraints $\mathcal{X}_i$ are nonempty closed convex sets.
\end{assu}
\begin{assu}\label{assu:KKT}
	Problem \eqref{DRSO} has a nonempty solution set, and there exists a KKT point satisfying \eqref{eq:KKT}.
\end{assu}

With the error function \eqref{eq:smalle}, DDRSM for solving optimization problem \eqref{DRSO} is presented as follows:

\begin{algorithm}[H]
	\caption{Distributed Douglas-Rachford splitting method for solving \eqref{DRSO}.}\label{Algo}
	\KwIn{Parameters $\rho\in(0,2)$, $\beta\in[\beta_{\min},\beta_{\max}]$, the initial points $\left(\begin{array}{c}x^0\\ \lambda^0\end{array}\right)\in\mathcal{X}\times\mathbb{R}^l$ and $\xi_i^0\in \partial f_i(x_i^0)$ for all $i=1,2,\dots,m$.}
	\For{$k=0,1,2,\cdots,$}
	{
	Compute $(\bar{e}_x^{k},e_\lambda^{k})$ by \eqref{eq:smalle}\;\label{eAlg}
	Update the step size $ \alpha_k $ through \eqref{eq:alpha}\;
	\For(\tcc*[f]{Possibly in parallel}){$i=1,2,\dots,m$}
	{$x_i^{k+1}=\operatorname{prox}_{\beta f_i}(x_i^k+\beta\xi_i^{k}-\rho\alpha_k\bar{e}_{x_i}^k)$\;
	$\xi_i^{k+1}=\xi_i^k+\frac{1}{\beta}(x_i^k-x_i^{k+1}-\rho\alpha_k\bar{e}_{x_i}^k)$\;
	}
	$\lambda^{k+1}=\lambda^k-\rho\alpha_k(e_\lambda^{k}-\beta A\bar{e}_x^{k})$\;
	}
\end{algorithm}

\begin{remark}
	The error functions in Algorithm \ref{Algo} are evaluated by substituting $ x^k $, $ \lambda^k $, and $ \bar{\lambda}^k $ into equations \eqref{eq:Ebeta}:
	\begin{equation}\label{eq:smalle}
		\begin{aligned}
			e_x^k &:= e_x(x^k, \lambda^k) = x^k - P_{\mathcal{X}} \left( x^k - \beta \left( \xi^k - A^{\top} \lambda^k \right) \right), \\
			e_\lambda^k &:= e_\lambda(x^k) = \beta \left( A x^k - b \right), \\
			\bar{e}_x^k &:= e_x(x^k, \bar{\lambda}^k) = x^k - P_{\mathcal{X}} \left( x^k - \beta \left( \xi^k - A^{\top} \bar{\lambda}^k \right) \right).
		\end{aligned}
	\end{equation}
	We denote $e_{x_i}^k$ and $\bar{e}_{x_i}^k$ as the $i$-th block of $e_x^k$ and $\bar{e}_x^k$, respectively.
	
	The calculation of $\alpha_k$ is determined as follows:
	\begin{equation}\label{eq:alpha}
		\begin{aligned}
			\varphi_{k} &=\left\|\bar{e}_{x}^{k}\right\|^{2}+\left\|e_{\lambda}^{k}\right\|^{2}-\beta \left\langle e_{\lambda}^{k}, A \bar{e}_{x}^{k}\right\rangle, \\
			\psi_{k} &=\left\|\bar{e}_{x}^{k}\right\|^{2}+\left\|e_{\lambda}^{k}-\beta A \bar{e}_{x}^{k}\right\|^{2}, \\
			\alpha_{k} &= \frac{\varphi_{k}}{\psi_{k}}.
		\end{aligned}
	\end{equation}
	It is important to note that if at least one of the error terms $ \bar{e}_{x}^k $ or $ e_{\lambda}^k $ is nonzero($\bar{e}_{x}^k\neq 0$ or $e_{\lambda}^k\neq 0$), then $ \psi_k $ is positive ($ \psi_k > 0 $). This ensures that $ \alpha_k $, which depends on $ \psi_k $, is well-defined. Moreover, by choosing the parameters properly (refer to Lemma \ref{lem:alpha}), we also have $\varphi_k > 0$; thus, the step size $\alpha_k$ is positive. If both $\bar{e}_{x}^k = 0$ and $e_{\lambda}^k = 0$, then we have already obtained the solution, and the algorithm can be terminated.
\end{remark}

\begin{remark}\label{remark2}
	Note that $x_i$-subproblems are proximal problems. Based on the first-order optimality condition \cite[Theorem~8.15]{rockafellar2009variational} for the proximal problem \eqref{eq:proximal}, we have
	\[
	- \frac{1}{\beta} \left( x_i^{k+1} - x_i^{k} - \beta \xi^{k}_i + \rho \alpha_k \bar{e}_{x_i}^k \right) \in \partial f_i \left( x^{k+1}_i \right).
	\]
	If we choose
	\[
	\xi^{k+1}_i=\xi^{k}_i +\frac{1}{\beta} \left( x_i^k-x_i^{k+1}-\rho\alpha_k \bar{e}_{x_i}^k \right),
	\]
	then we obtain $\xi^{k+1}_i \in \partial f_i \left( x^{k+1}_i\right)$ as well.
	
	Now, we can write two compact forms of DDRSM. The first is:
	\begin{equation}\label{disDR}
		\left\{\begin{array}{l}
			x_i^{k+1} = \operatorname{prox}_{\beta f_i} \left( x_i^{k} + \beta \xi^{k}_i - \rho \alpha_k \bar{e}_{x_i}^k \right), \\[1ex]
			x^{k+1}_i + \beta \xi_i^{k+1} = x_i^{k} + \beta \xi_i^{k} - \rho \alpha_k \bar{e}_{x_i}^k, \\[1ex]
			\lambda^{k+1} = \lambda^{k} - \rho \alpha_k \left( e_\lambda^k - \beta A \bar{e}_x^k \right).
		\end{array}\right.
	\end{equation}
	Here, the second line is obtained by using the previous property of the proximal operator.
	
	By utilizing the notations introduced in \eqref{DRSO3} for $ u $ and $ F(u) $, and denoting
	\begin{equation}\label{eq:compact}
		d_\beta(x^k, \bar{\lambda}^k) := \left( \begin{array}{c}
			e_{x_1} \left( x_1^k, \bar{\lambda}^k \right) \\
			\vdots \\
			e_{x_m} \left( x_m^k, \bar{\lambda}^k \right) \\
			e_\lambda \left( x^k \right) - \beta \sum_{i=1}^m A_i e_{x_i} \left( x_i^k, \bar{\lambda}^k \right)
		\end{array} \right), \quad
		w^k := \left( \begin{array}{c}
			x_1^k + \beta \xi_1^k \\
			\vdots \\
			x_m^k + \beta \xi_m^k \\
			\lambda^k
		\end{array} \right),
	\end{equation}
	we obtain a more compact form of DDRSM as
	\begin{equation}\label{eq:DDRSM_compact}
		w^{k+1} = w^k - \rho \alpha_{k} d_\beta \left( x^k, \bar{\lambda}^k \right).
	\end{equation}
\end{remark}

\section{Convergence analysis of DDRSM for weakly convex problems}\label{sec:converAnalysis}

In this section, we prove the linear convergence of Algorithm \ref{Algo} by introducing an assumption on the error bound.
In Remark \ref{remark2}, we demonstrated the relationship between DDRSM and the natural residual $E_\beta$. The error bound condition that we will introduce is also related to the natural residual.
We now present the following assumption:

\begin{assu}[Error Bound Condition]\label{assu:errorbound}
	Let $X^*,\Lambda^*$ be the solution sets for problem \eqref{DRSOVI}, and define
	\[
	W^* := \left\{ w^* := \left( \begin{array}{c}
		x^* + \beta \xi^* \\
		\lambda^*
	\end{array} \right) \ \Bigg| \ x^* \in X^*,\ \lambda^* \in \Lambda^*,\ \xi^* \in \partial f(x^*) \right\}.
	\]
	Then there exist constants $ \tau, \delta > 0 $ such that, for all $w=\left( \begin{array}{c}
		x + \beta \xi \\
		\lambda
	\end{array} \right)$, $x\in\mathcal{X}, \lambda\in\mathbb{R}^l$, $\xi\in\partial f(x)$, we have
	\begin{equation}\label{eq:errorbound}
		\operatorname{dist}(w,W^*) \leq \tau \|E_\beta(x,\bar{\lambda})\|,
	\end{equation}
	whenever $\|E_\beta(x,\bar{\lambda})\|\leq \delta$.
\end{assu}

\begin{remark}
	Assumption \ref{assu:errorbound} is widely recognized and has been introduced in the literature \cite{luo1992linear,luo1993error}. 
	From Lemma \ref{lem:naturalmap}, we know the equivalence between the solution sets and the zero points of $E_\beta$.
	The validity of the error bound condition is subsequently established for semistable variational inequalities under mild conditions (see \cite[Proposition~6.2.1]{facchinei2003finite}).
	To illustrate the validity of Assumption \ref{assu:errorbound}, we provide two examples from \cite{tseng2009coordinate} that satisfy Assumption \ref{assu:errorbound}.
	Consider a problem in the form of \eqref{DRSO}:
	\begin{itemize}
		\item $f_i(x)=h_i(M_ix_i)$, $\mathcal{X}_i$ are polyhedral sets, $h_i$ is twice continuously differentiable on $\mathbb{R}^{n_i}$ with a Lipschitz continuous gradient, and on any compact convex set, $h_i$ is strongly convex. This example has been analyzed in \cite[Lemma~6]{tseng2009coordinate};
		\item $f_i$ are quadratic functions, and $\mathcal{X}_i$ are polyhedral sets. This example has been analyzed in \cite[Theorem~4]{tseng2009coordinate}.
	\end{itemize}
\end{remark}

\begin{remark}
	The parameter $\delta$ serves as a constant that characterizes the ``locality'' of the error bound property within the problem. In practical scenarios, the error bound property can be satisfied over a large region or even globally for many functions, allowing for potentially large or arbitrarily large values of $\delta$ based on the chosen $\tau$.
	As a result, it becomes feasible to locate the initial point or an iteration point $x^{k_0}$ within the region where the error bound condition holds for some $\tau, \delta > 0$, which is the basis for the following convergence analysis.
\end{remark}

Next, to facilitate the subsequent proofs, we present two lemmas.
\begin{lem}\label{lem:alpha}
	Assume that $0< \beta <\frac{1}{\|A\|}$. Then the sequence $\left\{\alpha_k\right\}_{k=0}^{\infty}$ generated by Algorithm \ref{Algo} is bounded both above and below by positive constants, i.e.,
	\begin{equation}\label{eq:alphabound}
		\frac{1}{2} < \alpha_k \leq \frac{2+\beta\|A\|}{2(1-\beta\|A\|)}, \quad \forall k\geq 0.
	\end{equation}
\end{lem}
\begin{proof}
	By invoking the Cauchy-Schwarz inequality
	$$
	\langle a, b\rangle \leq \frac{\theta}{2}\|a\|^2+\frac{1}{2 \theta}\|b\|^2, \quad \forall a, b \in \mathbb{R}^n \text{ and } \theta>0,
	$$
	with the specifications $a:=e_\lambda^k$, $b:=A e_x^k$ and $\theta:=\|A\|$, we have
	$$
	\begin{aligned}
	\left\langle e_\lambda^k, A e_x^k\right\rangle & \leq \frac{\|A\|}{2}\left\|e_\lambda^k\right\|^2+\frac{1}{2\|A\|}\left\|A e_x^k\right\|^2 \\
	& \leq \frac{\|A\|}{2}\left(\left\|e_\lambda^k\right\|^2+\left\|e_x^k\right\|^2\right).
	\end{aligned}
	$$
	By substituting the above inequality into \eqref{eq:alpha}, we have
	\begin{equation}\label{eq:phivse}
		\begin{aligned}
			\varphi_k & =\left\|\bar{e}_x^k\right\|^2+\left\|e_\lambda^k\right\|^2-\beta \left\langle e_\lambda^k, A \bar{e}_x^k\right\rangle \\
			& \geq \frac{\left(2-\beta\|A\|\right)}{2}\left(\left\|e_\lambda^k\right\|^2+\left\|\bar{e}_x^k\right\|^2\right)=\frac{\left(2-\beta\|A\|\right)}{2}\|E_\beta(x^k,\bar{\lambda}^k)\|^2.
		\end{aligned}
	\end{equation}
	Similarly, we have
	\begin{equation}\label{eq:phiub}
		\varphi_k \leq \frac{\left(2+\beta\|A\|\right)}{2}\|E_\beta(x^k,\bar{\lambda}^k)\|^2.
	\end{equation}
	Since $ \beta < \frac{1}{\|A\|} $, we have $ \beta^2 \|A\bar{e}_x^k\|^2<\|\bar{e}_x^k\|^2 $, then 
	\begin{equation*}
		\begin{aligned}
			\psi_k &= \|\bar{e}_x^k\|^2 + \|e_\lambda^k - \beta A \bar{e}_x^k\|^2\\
			&= \|\bar{e}_x^k\|^2 + \|e_\lambda^k\|^2 - 2\beta  \left\langle e_\lambda^k, A \bar{e}_x^k\right\rangle + \beta^2 \|A\bar{e}_x^k\|^2 \\
			&< 2\left(\|\bar{e}_x^k\|^2 + \|e_\lambda^k\|^2 - \beta \left\langle e_\lambda^k, A \bar{e}_x^k \right\rangle \right)= 2\varphi_k.
		\end{aligned}
	\end{equation*}
	As a result, $\alpha_k=\frac{\varphi_k}{\psi_k}>\frac{1}{2}$.
	On the other hand, we have
	\begin{equation*}
		\begin{aligned}
			\psi_k &= \|\bar{e}_x^k\|^2 + \|e_\lambda^k\|^2 - 2\beta  \left\langle e_\lambda^k, A \bar{e}_x^k\right\rangle + \beta^2 \|A\bar{e}_x^k\|^2 \\
			&\geq (1-\beta\|A\|)\left(\|\bar{e}_x^k\|^2 + \|e_\lambda^k\|^2\right) + \beta^2 \|A\bar{e}_x^k\|^2 \\
			&\geq (1-\beta\|A\|)\|E_\beta(x^k,\bar{\lambda}^k)\|^2.
		\end{aligned}
	\end{equation*}
	Combining with \eqref{eq:phiub}, we have $\alpha_k\leq \frac{2+\beta\|A\|}{2(1-\beta\|A\|)}$.
\end{proof}

\begin{lem}
	Let $ \left\{w^k\right\}_{k=0}^{\infty} $ be defined as above, and let $ \left(\begin{array}{c}x^*\\ \lambda^*\end{array}\right) $ be a solution pair for \eqref{DRSO3}. We have
	$$ \left\langle w^{k}-w^*, d_\beta(x^k,\bar{\lambda}^k) \right\rangle \geq \varphi_k + \beta \left\langle x^k-x^*, \xi^k-\xi^*\right\rangle, $$
	where $ \varphi_k = \left\|e_{x}^{k}\right\|^{2} +\left\|e_{\lambda}^{k}\right\|^{2} -\beta \left\langle e_\lambda^k, A e_{x}^{k}\right\rangle $, $w^*=x^*+\beta\xi^*$ and $ \xi^*\in \partial f(x^*) $. \label{lem:crossterm}
\end{lem}

\begin{proof}
	The solution for \eqref{DRSO3} should also satisfy \eqref{DRSOVI}, which means that the following variational inequality holds:
	\begin{equation*}\label{eq:VI}
		\left\langle x^\prime-x^*,\;\xi^*-A^\top\lambda^*\right\rangle\geq 0, \quad \forall x^\prime \in \mathcal{X}.
	\end{equation*}
	Take $ x^\prime:=P_{\mathcal{X}}(x^k-\beta(\xi^k-A^\top\bar{\lambda}^k))=x^k-e_{x}(x^k,\bar{\lambda}^k) $. 
	Then we have 
	\begin{equation} \label{eq:exVI}
		\left\langle x^k-e_{x}(x^k,\bar{\lambda}^k)-x^*,\;\xi^*-A^\top\lambda^*\right\rangle\geq 0.
	\end{equation}
	According to the projection property, since $x^\prime$ is a projection onto $\mathcal{X}$ and $x^*\in\mathcal{X}$, we have
	\begin{equation*}
		\left\langle \left(x^k-\beta(\xi^k-A^\top\bar{\lambda}^k)\right)-x^\prime, x^*-x^\prime \right\rangle \leq 0, 
	\end{equation*}
	which is
	\begin{equation}\label{eq:exproj}
		\left\langle e_{x}(x^k,\bar{\lambda}^k) - \beta(\xi^k-A^\top\bar{\lambda}^k),\; x^*-x^k+e_{x}(x^k,\bar{\lambda}^k)\right\rangle \leq 0.
	\end{equation}
	Multiplying inequality \eqref{eq:exVI} by $ \beta $, and then adding \eqref{eq:exproj} to it, we obtain
	\begin{equation} 
		\left\langle x^k-x^*-e_{x}(x^k,\bar{\lambda}^k),\; e_{x}(x^k,\bar{\lambda}^k)-\beta(\xi^k-\xi^*)+\beta A^\top(\bar{\lambda}^k-\lambda^*)\right\rangle\geq 0.
	\end{equation}
	Rearranging the terms, we could deduce
	\begin{equation}
		\begin{aligned}
			&\left\langle w^{k}-w^*,\;d_\beta(x^k,\bar{\lambda}^k)\right\rangle \\
			=&\left\langle z^k-z^*,\; e_{x}(x^k,\bar{\lambda}^k)\right\rangle + \beta \left\langle {\lambda}^k-\lambda^*,\; A(x^k-x^*-e_{x}(x^k,\bar{\lambda}^k))\right\rangle \\
			\geq & \|e_{x}(x^k,\bar{\lambda}^k)\|^2 + \beta\left\langle x^k-x^*,\;\xi^k-\xi^*\right\rangle + \beta\left\langle \lambda^k - \bar{\lambda}^k,\; A(x^k-x^*-e_{x}(x^k,\bar{\lambda}^k))\right\rangle \\
			=& \|e_{x}(x^k,\bar{\lambda}^k)\|^2 + \|e_\lambda(x^k)\|^2 - \beta \left\langle e_\lambda(x^k),\; A e_{x}(x^k,\bar{\lambda}^k)\right\rangle + \beta \left\langle x^k-x^*,\;\xi^k-\xi^*\right\rangle\\
			=& \varphi_k + \beta \left\langle x^k-x^*,\;\xi^k-\xi^*\right\rangle,
		\end{aligned}
	\end{equation}
	which finishes the proof.
\end{proof}

\begin{thm}\label{thm:convergence}
	Suppose Assumptions \ref{assu:f} and \ref{assu:sets} hold,
	and suppose the error bound Assumption \ref{assu:errorbound} holds for some $ \tau,\delta>0$.
	Let $c_0>0$ be an upper bound of all weak convexity coefficients of $f_i$, or in other words, $f$ is $c_0$-weakly convex.
	Let $\beta\in\left( 0, \operatorname{min}(\frac{1}{2c_0},\frac{1}{\|A\|+c_0},\frac{2}{\|A\|+2c_0\tau^2}) \right)$, and let $\left\{\left(\begin{array}{c}x^k\\ \lambda^k\end{array}\right)\right\}_{k=0}^{\infty}$ be a sequence generated by Algorithm \ref{Algo}.
	Then if there exist $x^{k_0}$ and $\lambda^{k_0}$ such that $E_\beta(x^{k_0},\bar{\lambda}^{k_0})\leq\delta$, we have	
	$$ \lim_{k\rightarrow\infty} \|E_\beta(x^k,\bar{\lambda}^k)\|^2 =0. $$	
	And the sequence $\left\{\left(\begin{array}{c}x^k\\ \lambda^k\end{array}\right)\right\}_{k=0}^{\infty}$ converges to a solution of \eqref{DRSOVI}.
\end{thm}

\begin{proof}
	Since $c_0$ is the upper bound of all weak convexity coefficients of $f_i$,
	by substituting $x=x^k$ and $y=x^*$ into equation \eqref{eq:weakconvex}, we have
	\begin{equation*}
			-\left\langle x^k-x^*,\; \xi^k-\xi^*\right\rangle \leq c_0\|x^k-x^*\|^2 \leq c_0\|x^k-x^*\|^2 + c_0\beta^2\|\xi^k - \xi^*\|^2.
	\end{equation*}
	Since $\beta\in(0,\frac{1}{2c_0})$, we set $c=\frac{c_0}{1-2c_0\beta}>c_0$. 
	Then we have
	\begin{equation*}
		-(2c\beta+1)\left\langle x^k-x^*,\; \xi^k-\xi^*\right\rangle \leq c\|x^k-x^*\|^2 + c\beta^2\|\xi^k - \xi^*\|^2.
	\end{equation*}
	Rearranging them, we get
	\begin{equation*}
		-\left\langle x^k-x^*,\; \xi^k-\xi^*\right\rangle \leq c\|x^k-x^*\|^2 + c\beta^2\|\xi^k - \xi^*\|^2 + 2c\beta\left\langle x^k-x^*,\; \xi^k-\xi^*\right\rangle=c\|(x^k-x^*)+\beta(\xi^k - \xi^*)\|^2.
	\end{equation*}
	Thus, by taking $w=w^k=\left(\begin{array}{c}x^k+\beta\xi^k\\ \lambda^k\end{array}\right)$ and $ w^* = w^*_k = \left(\begin{array}{c}x^*_k+\beta\xi^*_k\\ \lambda^*_k\end{array}\right) = P_{W^*}(w^k)$, we have
	\begin{equation*}
		\begin{aligned}
			-\left\langle x^k-x^*_k,\; \xi^k-\xi^*_k\right\rangle &\leq c\|(x^k-x^*_k)+\beta(\xi^k - \xi^*_k)\|^2 + c\|\lambda^k - \lambda^*_k\|^2\\
			&= c\;\operatorname{dist}^2(w^k,W^*).
		\end{aligned}
	\end{equation*}
	The inequality arises from the weak convexity of $f_i$.
	By Assumption \ref{assu:errorbound}, we have for all $k\geq k_0$,
	\begin{equation*}
		\operatorname{dist}(w^k,W^*) \leq \tau \|E_\beta(x^k,\bar{\lambda}^k)\|.
	\end{equation*}
	Without loss of generality, we assume $k_0=0$. Thus,
	\begin{equation}\label{eq:crossvsE}
		-\left\langle x^k-x^*_k,\; \xi^k-\xi^*_k\right\rangle\leq c\tau^2 \|E_\beta(x^k,\bar{\lambda}^k)\|^2.
	\end{equation}
	This, together with Lemma \ref{lem:alpha} and \ref{lem:crossterm}, yields
	\begin{equation}\label{eq:crossterm}
		\left\langle w^{k}-w^*_k,\; d_\beta(x^k,\bar{\lambda}^k)\right\rangle \geq \frac{2-\beta\|A\|-2\beta c\tau^2}{2}\|E_\beta(x^k,\bar{\lambda}^k)\|^2.
	\end{equation}
	
	Since $c>c_0$, we have $\beta \in \left(0,\frac{1}{\|A\|+c}\right)$.
	This means that the right-hand side of inequality \eqref{eq:crossterm} is positive, and we can get
	\begin{equation}
		\begin{aligned}
			\|w^{k+1}-w^*_k\|^2&\leq \|w^k -\rho\alpha_k d_\beta(x^k,\bar{\lambda}^k) -w^*_k\|^2\\
			&= \|w^k - w^*_k\|^2 - 2\rho\alpha_k\left\langle w^k  -w^*_k,\; d_\beta(x^k,\bar{\lambda}^k)\right\rangle  + \rho^2\alpha_k^2\| d_\beta(x^k,\bar{\lambda}^k)\|^2\\
			&\leq \|w^k - w^*_k\|^2 -  2\rho\alpha_k\left(\varphi_k + \beta\left\langle x^k-x^*_k,\; \xi^k-\xi^*_k\right\rangle\right)  - \rho^2\alpha_k^2\| d_\beta(x^k,\bar{\lambda}^k)\|^2 \\
			&\leq \|w^k - w^*_k\|^2 -  2\rho\alpha_k\varphi_k + \rho^2\alpha_k^2\psi_k + 2\rho\alpha_k\beta c\tau^2 \|E_\beta(x^k,\bar{\lambda}^k)\|^2,
		\end{aligned}
	\end{equation}
	where the second inequality comes from Lemma \ref{lem:crossterm} and the last one comes from \eqref{eq:crossvsE} and the definition of $ \psi_k $ in \eqref{eq:alpha}.
	Since $ \beta \in \left(0,\frac{1}{\|A\|+c}\right) $, according to Lemma \ref{lem:alpha}, if we choose $ \alpha_k:=\frac{\varphi_k}{\psi_k} $, we know that $ \alpha_k>\frac{1}{2} $. Then we have
	\begin{equation*}
		\begin{aligned}
			\|w^{k+1}-w^*_k\|^2&\leq \|w^k - w^*_k\|^2 -  2\rho\alpha_k\varphi_k + \rho^2\alpha_k\varphi_k + 2\rho\alpha_k\beta c \tau^2 \|E_\beta(x^k,\bar{\lambda}^k)\|^2\\
			&=\|w^k - w^*_k\|^2 - \rho(2-\rho)\alpha_k \varphi_k + 2\rho\alpha_k\beta c \tau^2 \|E_\beta(x^k,\bar{\lambda}^k)\|^2\\
			&=\|w^k - w^*_k\|^2 - \rho\alpha_k \left((2-\rho) \varphi_k - 2\beta c \tau^2 \|E_\beta(x^k,\bar{\lambda}^k)\|^2\right)\\
			&\leq\|w^k - w^*_k\|^2 - \rho\alpha_k \left((2-\rho) \frac{(2-\beta\|A\|)}{2} - 2\beta c \tau^2\right) \|E_\beta(x^k,\bar{\lambda}^k)\|^2\\
			&\leq\|w^k - w^*_k\|^2 - \rho\left(2-\beta\|A\|-2\beta c \tau^2 - \frac{(2-\beta\|A\|)\rho}{2}\right)\|E_\beta(x^k,\bar{\lambda}^k)\|^2.
		\end{aligned}
	\end{equation*}
	The second inequality comes from \eqref{eq:phivse}. The last one holds due to $ \alpha_k>\frac{1}{2} $. The coefficient of the $ \|E_\beta(x^k,\bar{\lambda}^k)\|^2 $ term is a constant $C = 2-\beta\|A\|-2\beta c \tau^2 - \frac{(2-\beta\|A\|)\rho}{2}$. 
	
	Also, from $c>c_0$, we know that $\beta\in\left(0,\frac{2}{\|A\|+2c\tau^2}\right)$. Therefore, we have $2-\beta\|A\|-2\beta c \tau^2>0$ and $2-\beta\|A\|>0$. If we choose $\rho$ to be small enough, the constant $C$ will be positive. Then we have
	\begin{equation}\label{eq:convergeC0}
		\begin{aligned}
			\operatorname{dist}^2(w^{k+1},W^*)&=\|w^{k+1}-w^*_{k+1}\|^2\leq\|w^{k+1}-w^*_k\|^2\\
			&\leq\|w^k - w^*_k\|^2 - C \|E_\beta(x^k,\bar{\lambda}^k)\|^2\\
			&=\operatorname{dist}^2(w^{k},W^*)- C \|E_\beta(x^k,\bar{\lambda}^k)\|^2.
		\end{aligned}
	\end{equation}
	
	This immediately leads to $ \lim_{k\rightarrow\infty} \|E_\beta(x^k,\bar{\lambda}^k)\|^2 =0 $.
	Recall the iteration formulation in \eqref{eq:DDRSM_compact}, we have
	\[
	\|w^{k+1} - w^k\|^2 = \rho^2\alpha_k^2 \|d_\beta(x^k,\bar{\lambda}^k)\|^2 = \rho^2 \alpha_k \cdot \alpha_k \|d_\beta(x^k,\bar{\lambda}^k)\|^2 = \rho^2 \alpha_k \cdot \alpha_k \psi_k = \rho^2 \alpha_k \varphi_k,
	\]
	where the second equality comes from the definition of $d_\beta(x^k,\bar{\lambda}^k)$ in \eqref{eq:compact} and 
	the third and fourth equalities come from the definition of $\varphi_k, \psi_k $ and $ \alpha_k $ in \eqref{eq:alpha}.
	By the upper and lower bounds of $\varphi_k$ in \eqref{eq:phivse} and \eqref{eq:phiub}, we have
	\begin{equation*}
		\frac{2 - \beta\|A\|}{2} \rho^2 \alpha_k \|E_\beta(x^k,\bar{\lambda}^k)\|^2 \leq \|w^{k+1} - w^k\|^2  \leq \frac{2 + \beta\|A\|}{2} \rho^2 \alpha_k \|E_\beta(x^k,\bar{\lambda}^k)\|^2.
	\end{equation*}
	Using the bounds of $\alpha_k$ given in \eqref{eq:alphabound} of Lemma \ref{lem:alpha}, we have
	\begin{equation*}
		\rho^2 \cdot \frac{2 - \beta\|A\|}{4}\|E_\beta(x^k,\bar{\lambda}^k)\|^2  < \|w^{k+1} - w^k\|^2 \leq \rho^2 \cdot \frac{2 +\beta\|A\|}{2} \frac{2+\beta\|A\|}{2(1-\beta\|A\|)} \|E_\beta(x^k,\bar{\lambda}^k)\|^2.
	\end{equation*}
	Since $\beta < \frac{1}{\|A\|}$, we can set
	\begin{equation*}
		c_1^2 = \rho^2 \cdot \frac{2 - \beta\|A\|}{4}, \quad \text{and} \quad c_2^2 = \rho^2 \cdot \frac{2 +\beta\|A\|}{2} \frac{2+\beta\|A\|}{2(1-\beta\|A\|)}.
	\end{equation*}
	Then, using \eqref{eq:convergeC0}, we have
	\begin{equation*}
		\begin{aligned}
			\|w^{k+1} - w^k\|^2 & \leq c_2^2 \|E_\beta(x^k,\bar{\lambda}^k)\|^2 \leq \frac{c_2^2}{C} \left( \operatorname{dist}^2(w^k,W^*) - \operatorname{dist}^2(w^{k+1},W^*) \right)\\
			&\leq \frac{c_2^2}{C} \left( \operatorname{dist}(w^k,W^*) - \operatorname{dist}(w^{k+1},W^*) \right) \cdot 2 \operatorname{dist}(w^k,W^*).
		\end{aligned}
	\end{equation*}
	And for all $ k \geq k_0 $, according to Assumption \ref{assu:errorbound}, we have $ \operatorname{dist}(w^k,W^*) \leq \tau  \|E_\beta(x^k,\bar{\lambda}^k)\|$, which leads to
	\begin{equation*}
		\begin{aligned}
			\|w^{k+1} - w^k\|^2 & \leq \frac{2\tau c_2^2}{C} \left( \operatorname{dist}(w^k,W^*) - \operatorname{dist}(w^{k+1},W^*) \right)  \|E_\beta(x^k,\bar{\lambda}^k)\|\\
			&< \frac{2\tau c_2^2}{Cc_1} \left( \operatorname{dist}(w^k,W^*) - \operatorname{dist}(w^{k+1},W^*) \right)  \| w^{k+1} - w^k\|.
		\end{aligned}
	\end{equation*}
	Thus, we have
	\begin{equation}\label{eq:wk_caythy}
		\|w^{k+1} - w^k\| < \frac{2\tau c_2^2}{Cc_1} \left( \operatorname{dist}(w^k,W^*) - \operatorname{dist}(w^{k+1},W^*) \right).
	\end{equation}
	Summing inequality \eqref{eq:wk_caythy} over $ k $ from 0 to $ N-1 $, we obtain
	\begin{equation}\label{eq:wk_caythy_sum}
		\begin{aligned}
			\sum_{k=0}^{N-1} \|w^{k+1} - w^k\| &\leq \frac{2\tau c_2^2}{Cc_1} \sum_{k=0}^{N-1} \left( \operatorname{dist}(w^k,W^*) - \operatorname{dist}(w^{k+1},W^*) \right) \\
			&= \frac{2\tau c_2^2}{Cc_1} \left( \operatorname{dist}(w^0,W^*) - \operatorname{dist}(w^{N},W^*) \right).
		\end{aligned}
	\end{equation}
	As $N$ goes to infinity, the left-hand side converges to a finite value.
	Thus, $ \left\{ w^k \right\}_{k=0}^{\infty} $ is a Cauchy sequence, and its limit $ w^\infty=(x^\infty+\xi^\infty,\lambda^\infty) $ satisfies $ E_\beta(x^\infty,\bar{\lambda}^\infty) = 0 $ which indicates that $ (x^\infty,\lambda^\infty) $ is a solution to \eqref{DRSOVI}.
\end{proof}

According to Lemma \ref{lem:naturalmap}, we have proven that the algorithm converges. 
In fact, combining \eqref{eq:convergeC0} with Assumption \ref{assu:errorbound}, we have
\begin{equation}\label{eq:linconvergence}
	\begin{aligned}
		\operatorname{dist}^2(w^{k+1},W^*)&\leq\operatorname{dist}^2(w^k,W^*) - C \|E_\beta(x^k,\bar{\lambda}^k)\|^2\\
		&\leq (1 - \frac{C}{\tau^2})\operatorname{dist}^2(w^k,W^*)\\
		&=C_q^2\operatorname{dist}^2(w^k,W^*).
	\end{aligned}
\end{equation}
Using triangle inequality and \eqref{eq:wk_caythy_sum}, we have
	\begin{equation*}
		\begin{aligned}
			\| w^{M} - w^* \| &\leq \sum_{k=M}^{\infty} \| w^{k+1} - w^{k} \| 
			\leq \frac{2\tau c_2^2}{Cc_1} \operatorname{dist}(w^M,W^*) \\
			&\leq \frac{2\tau c_2^2}{Cc_1} C_q^{M-N} \operatorname{dist}(w^N,W^*)
			\leq \frac{2\tau c_2^2}{Cc_1}C_q^{M-N} \| w ^{N} - w^* \|.
		\end{aligned}
	\end{equation*}
Since then, we have proven the linear convergence of the algorithm with linear convergence rate $C_q$.

\begin{thm} \label{thm:linear}
	Let $\left\{\left(\begin{array}{c}x^k\\ \lambda^k\end{array}\right)\right\}_{k=0}^{\infty}$ be a sequence generated by Algorithm \ref{Algo}. 
	With the same parameter setting of $\beta$, and under the same assumptions as Theorem \ref{thm:convergence} with respect to $ c_0 $, $ \tau $, and $ \delta $,
	if there exist $x^{k_0}$ and $\lambda^{k_0}$ such that $E_\beta(x^{k_0},\bar{\lambda}^{k_0})\leq\delta$, 
	then there exists $0 < C_q < 1$, such that for all $ k \geq k_0 $, the following inequality holds:
	$$ \operatorname{dist}(w^{k+1},W^*) \leq C_q \operatorname{dist}(w^k,W^*). $$
	This implies that the algorithm converges linearly, specifically,
	$$ \| w ^{M} - w^* \| \leq  \frac{2\tau c_2^2}{Cc_1} C_q^{M-N} \| w ^{N} - w^* \|. $$
	for all $M>N\geq k_0$, where $w^*$ is the limit point of the sequence as well as a solution to \eqref{DRSOVI}.
\end{thm}

\section{Numerical experiments} \label{numericalExp}

In this section, we aim to validate the effectiveness of our proposed method through numerical experiments. Nonconvex models often outperform their convex counterparts in practical scenarios, particularly when modeling sparsity and low-rank structures. While DDRSM, as introduced by He and Han in \cite{he2016distributed}, has demonstrated superiority over methods like ADMM in solving convex problems, it remains uncertain whether DDRSM maintains this advantage in nonconvex scenarios. Thus, our objectives in the numerical experiments are twofold: to confirm the efficacy of DDRSM in handling nonconvex models, and to investigate whether DDRSM's performance advantage over ADMM and other ALM-based methods persists in these nonconvex problems.

We begin by addressing a compressed sensing problem related to signal reconstruction. Subsequently, we investigate the robust alignment of linearly correlated images using sparse and low-rank decomposition techniques. The experiments were conducted on a desktop computer with a 2.5 GHz Intel(R) Core(TM) i5-12400F processor and 32~GB of memory, and the algorithms were implemented using MATLAB 2023a.

\subsection{Compressed sensing for signal with $\ell_{1/2}$-norm regularization}

First, we conduct a simple numerical experiment on the LASSO model for compressed sensing.
Compressed sensing aims to reconstruct a sparse signal from a limited number of observations.
The observed data can be represented as follows:
\begin{equation*}
	v = M R u + \epsilon,
\end{equation*}
where $v \in \mathbb{R}^m$ is the observed data, $M \in \mathbb{R}^{m \times n}$ is a sparse measurement matrix (with $n \ll m$), $R\in \mathbb{R}^{n \times n}$ is an invertible transform (typically an orthogonal matrix such as the discrete Fourier transform or wavelet transform), $u \in \mathbb{R}^n$ is the signal to be reconstructed, and $\epsilon\in \mathbb{R}^m$ accounts for noise. Compressed sensing is applicable when the signal $x = Ru$ is sparse. The objective is to recover the signal $x^*$ by solving the following optimization problem:
\begin{equation*}
	x^* = \underset{x}{\operatorname{argmin}}\|x\|_0, \quad \text{s.t.}\; v = Mx,
\end{equation*}
	where $\|\cdot\|_0$ denotes the $\ell_0$ norm, representing the number of nonzero entries in a vector.

The LASSO method is a popular approach for reconstructing sparse signals and is widely used in conjunction with $\ell_{1/2}$ norm regularization.
In addition to the $\ell_{1/2}$ norm, LASSO can incorporate nonconvex norms to model data more effectively.
The LASSO model can be formulated as a constrained minimization problem:
\begin{equation*}
	\min _{x}\left\{\frac{1}{2\delta}\|Mx-v\|_{2}^2+\|x\|_{q}^q \right\},
\end{equation*}
where the parameter $\delta > 0$ balances noise reduction and sparsity, and $\|\cdot\|_{q}$ represents the $\ell_q$ quasi-norm with $0 < q < 1$.

By introducing a variable $y$ such that $Mx = y$, the problem can be reformulated as:
\begin{equation}\label{eq:cs}
	\min_{x}\left\{ \|x\|_{q}^q + \frac{1}{2\delta}\|y - v\|_{2}^2 \ \bigg|\ Mx = y \right\}.
\end{equation}
Before proceeding with the experiments, we need to address the nonconvex term $\|x\|_{q}^q$ in the objective function.

Given the inherent challenges in solving the original objective function due to its nonconvex, nonsmooth, and non-Lipschitz nature—which further fails to meet the weak convexity required by our algorithm—we pursue an alternative strategy by focusing on a relaxed, smooth problem. We introduce $r^q_{\varepsilon}(x)$ to represent the smoothed approximation of $\|x\|_{q}^q\;(q\in(0,1))$ in model \eqref{eq:cs}, detailed by
\begin{equation}\label{eq:smoothing}
    r^q_{\varepsilon}(x) = \sum_{i=1}^{n} r^q_{\varepsilon,i}(x_i) = \sum_{i=1}^{n} \left\{
    \begin{array}{ll}
        |x_i|^q, & |x_i| > \varepsilon, \\[1ex]
        \dfrac{1}{2} q \varepsilon^{q-2} x_i^2 + \dfrac{q - 2}{2} \varepsilon^q, & |x_i| \leq \varepsilon,
    \end{array}\right.
\end{equation}
where $\varepsilon$ is a small positive constant. 
This approximation ensures that $r^q_{\varepsilon,i}(x_i)$ ranges from $[\frac{q-2}{2}\varepsilon^q,+\infty)$, and its derivative, $(r^q_{\varepsilon,i}(x_i))^\prime$, falls within $[-q\varepsilon^{q-1},q\varepsilon^{q-1}]$. 
Such a smoothing technique, leveraging a quadratic smoothing function, facilitates the derivation of a closed-form solution for the proximal operator of $r^q_{\varepsilon}(x)$, a strategy echoed in previous studies \cite{wu2014batch, fan2014convergence}.

For $q=\frac{1}{2}$, the proximal operator of $\|x\|_{q}^q$ has a closed-form solution similar to the soft-thresholding operator, as discussed in \cite{xu2012l12}. The $i$-th element of the operator is given by:
\begin{equation}\label{eq:prox1/2}
	\left( \operatorname{prox}_{\beta \|\cdot\|_{1/2}^{1/2}}(x) \right)_i = \underset{y}{\operatorname{argmin}}\left\{ |y|^{1/2} + \frac{1}{2\beta} \|y - x_i\|_{2}^2 \right\}
	= \left\{
	\begin{array}{ll}
		f_d(x_i), & |x_i| > \dfrac{\sqrt[3]{54}}{4} \beta^{2/3}, \\[2ex]
		0, & |x_i| \leq \dfrac{\sqrt[3]{54}}{4} \beta^{2/3},
	\end{array}\right.
\end{equation}
where $f_d(x_i) = \dfrac{2}{3} x_i \left( 1 + \cos \left( \dfrac{2\pi}{3} - \dfrac{2}{3} \phi_\beta\left( x_i \right) \right) \right)$ with $\phi_\beta\left( x_i \right) = \arccos \left( \dfrac{\beta}{8} \left( \dfrac{\left| x_i \right|}{3} \right)^{-3/2} \right)$.

To calculate the proximal operator of $r^q_{\varepsilon}(x)$ and $|x|^q$, we need to analyze the objective function of the proximal problem. For convenience, we denote $R_{x_i}(y) = \dfrac{1}{2\beta}(y - x_i)^2 + r^q_{\varepsilon,i}(y)$ and $T_{x_i}(y) = \dfrac{1}{2\beta}(y - x_i)^2 + |y|^q$. Without loss of generality, we assume $x_i > 0$.

When $|y| \geq \varepsilon$, $R_{x_i}(y) = T_{x_i}(y)$. Since we assume $x_i > 0$, we have $T_{x_i}(\varepsilon) \leq T_{x_i}(-\varepsilon)$. The minimum of $R_{x_i}$ is attained at either $\varepsilon$ or the local minima of $T_{x_i}$. When $|y| \leq \varepsilon$, $R_{x_i}(y)$ is a quadratic function, so the minimum is attained at $\dfrac{x_i}{1 + \beta q \varepsilon^{q - 2}}$ if $\dfrac{x_i}{1 + \beta q \varepsilon^{q - 2}} < \varepsilon$; otherwise, it is attained at $\varepsilon$.

According to \cite[Theorem 1]{xu2012l12}, the monotonicity of the function $T_{x_i}$ differs depending on whether $x_i \leq \dfrac{3}{4} \beta^{2/3}$. Therefore, we proceed by discussing the cases accordingly:

\begin{itemize}
	\item \textbf{Case 1:} $x_i > \dfrac{3}{4} \beta^{2/3}$. According to \cite[Theorem 1]{xu2012l12}, there is a local minimum of $T_{x_i}$ at $f_d(x_i)$ other than zero. When $|y| \geq \varepsilon$, if $\varepsilon \leq f_d(x_i)$, the local minima of $R_{x_i}$ and $T_{x_i}$ coincide at $f_d(x_i)$. Otherwise, the minimum of $R_{x_i}$ is attained at $\varepsilon$, since $R_{x_i}$ is monotonic when $y > \varepsilon \geq f_d(x_i)$. When $|y| \leq \varepsilon$, the minimum is attained at $\dfrac{x_i}{1 + \beta q \varepsilon^{q - 2}}$ or $\varepsilon$.

	\item \textbf{Case 2:} $x_i \leq \dfrac{3}{4} \beta^{2/3}$. The derivative of $\dfrac{1}{2\beta}(y - x_i)^2$ no longer intersects with the derivative of $|y|^q$, implying that $T_{x_i}$ is monotonic on both sides of the origin. Thus, $R_{x_i}$ is monotonic when $y \geq \varepsilon$, and the minimum is attained at $\dfrac{x_i}{1 + \beta q \varepsilon^{q - 2}}$ or $\varepsilon$.
\end{itemize}

As a result, the proximal operator of $r^q_{\varepsilon,i}(x_i)$ can be expressed as:
\begin{equation}\label{eq:proxq}
	\left(\operatorname{prox}_{\beta r^q_{\varepsilon,i}}(x)\right)_i = \left\{\begin{array}{ll}
		\underset{y\in\left\{f_{d}(x_i),\;\mathrm{sgn}(x_i)\min\{|\frac{x_i}{1+\beta q \varepsilon^{q-2}}|,|\varepsilon|\}\right\}}{\arg\min} R_{x_i}(y), & |x_i| > \frac{3}{4}\beta^{(2/3)},\\
		\mathrm{sgn}(x_i)\min\{|\frac{x_i}{1+\beta q \varepsilon^{q-2}}|,|\varepsilon|\}, & |x_i| \leq \frac{3}{4}\beta^{(2/3)}.
	\end{array}\right.
\end{equation}
When $|x_i| > \dfrac{3}{4} \beta^{2/3}$, we evaluate both candidate values of $y$ and choose the $y$ that minimizes $R_{x_i}(y)$.

Our model can now be rewritten as:
\begin{equation}\label{eq:cs2}
	\min_{x}\left\{ r^q_{\varepsilon}(x) + \frac{1}{2\delta}\|y - v\|_{2}^2 \ \bigg|\ Mx = y \right\}.
\end{equation}

The model in equation \eqref{eq:cs2} can be expressed in the form of \eqref{DRSO} with the following specifications:
\[
	x_1 = x,\quad x_2 = y;\quad
	f_1(x_1) = r^q_{\varepsilon}(x),\quad f_2(x_2) = \dfrac{1}{2\delta}\|x_2 - v\|_{2}^2;\quad
	A = \begin{bmatrix} M & -I \end{bmatrix},\quad b = 0.
\]
There are no abstract constraints, so $\mathcal{X} = \mathbb{R}^{m+n}$.
Therefore, we can calculate the error functions in equation \eqref{eq:smalle} as follows:
\begin{equation*}
	\begin{aligned}
		e_\lambda^k &= \beta (Mx^k - y^k), \\
		\bar{e}_x^k &= \beta \left( \xi_x^k - M^\top (\lambda^k - e_\lambda^k) \right), \\
		\bar{e}_y^k &= \beta \left( \xi_y^k + (\lambda^k - e_\lambda^k) \right).
	\end{aligned}
\end{equation*}

With these specifications, the subproblems can be summarized as follows:
\begin{itemize}
	\item \textbf{The $x_1$-subproblem ($x$-subproblem):}
		\begin{equation*}
			\begin{aligned}
			x^{k+1}
			&= \underset{x \in \mathbb{R}^{n}}{\operatorname{argmin}}\left\{ r^q_{\varepsilon}(x) + \frac{1}{2\beta} \left\| x - \left( x^k + \beta \xi_x^k - \rho_k \bar{e}_{x}^{k} \right) \right\|_{2}^{2} \right\} \\
			&= \operatorname{prox}_{\beta r^q_{\varepsilon}}\left( x^{k} + \beta \xi_x^k - \rho_{k} \bar{e}_{x}^{k} \right).
			\end{aligned}
		\end{equation*}
	\item \textbf{The $x_2$-subproblem ($y$-subproblem):}
		\begin{equation*}
			\begin{aligned}
			y^{k+1}
			&= \underset{y \in \mathbb{R}^{m}}{\operatorname{argmin}}\left\{ \frac{1}{2\delta}\|y - v\|_{2}^{2} + \frac{1}{2\beta} \left\| y - \left( y^{k} + \beta \xi_y^k - \rho_{k} \bar{e}_{y}^{k} \right) \right\|_{2}^{2} \right\} \\
			&= \frac{\delta \left( y^{k} + \beta \xi_y^k - \rho_{k} \bar{e}_{y}^{k} \right) + \beta v}{\delta + \beta}.
			\end{aligned}
		\end{equation*}
	\item \textbf{The $ \lambda $-update:}
		\begin{equation*}
			\lambda^{k+1} = \lambda^k - \rho_{k} \left( e_\lambda^k - \beta \left( M \bar{e}_x^k - \bar{e}_y^k \right) \right).
		\end{equation*}
\end{itemize}

\begin{table}[htbp]
	\centering
	\caption{Performance comparison between DDRSM and ADMM under nonconvex settings}
	\begin{tabular*}{\textwidth}{@{\extracolsep{\fill}}lcccccc}
		\toprule
		& \multicolumn{2}{c}{\textbf{Problem Setting}} & \multicolumn{4}{c}{\textbf{Performance Metrics}} \\
		\cmidrule(lr){2-3} \cmidrule(lr){4-7}
		\textbf{Algorithm} & $(m,n)$ & Sparsity & PSNR & CPU Time & Iterations \\
		\midrule
		DDRSM & \multirow{2}{*}{$(1500,1000)$} & \multirow{2}{*}{0.02} & \textbf{62.79} & \textbf{0.4178s} & \textbf{23} \\
		ADMM & & & 62.78 & 0.6105s & 57 \\
		\midrule
		DDRSM & \multirow{2}{*}{$(3000,1000)$} & \multirow{2}{*}{0.02} & \textbf{63.46} & \textbf{0.7983s} & \textbf{23} \\
		ADMM & & & 63.46 & 0.9376s & 52 \\
		\midrule
		DDRSM & \multirow{2}{*}{$(1500,1000)$} & \multirow{2}{*}{0.06} & 67.53 & \textbf{0.4330s} & \textbf{25} \\
		ADMM & & & \textbf{67.63} & 0.5607s & 60 \\
		\midrule
		DDRSM & \multirow{2}{*}{$(1500,1000)$} & \multirow{2}{*}{0.12} & \textbf{62.29} & 0.6892s & \textbf{40} \\
		ADMM & & & 62.28 & \textbf{0.6121s} & 63 \\
		\bottomrule
	\end{tabular*}
	\label{tab:1}
\end{table}

\begin{figure}[htbp]
	\centering
	\begin{subfigure}[b]{0.45\textwidth}
		\centering
		\includegraphics[height=2in]{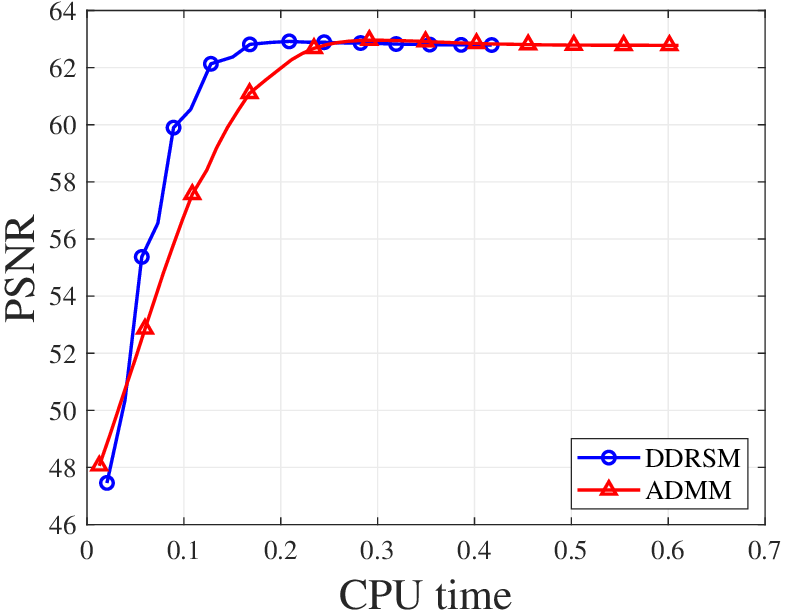}
		\caption{$m=1500, n=1000, \text{sparsity}=0.02$}
	\end{subfigure}
	\hfil
	\begin{subfigure}[b]{0.45\textwidth}
		\centering
		\includegraphics[height=2in]{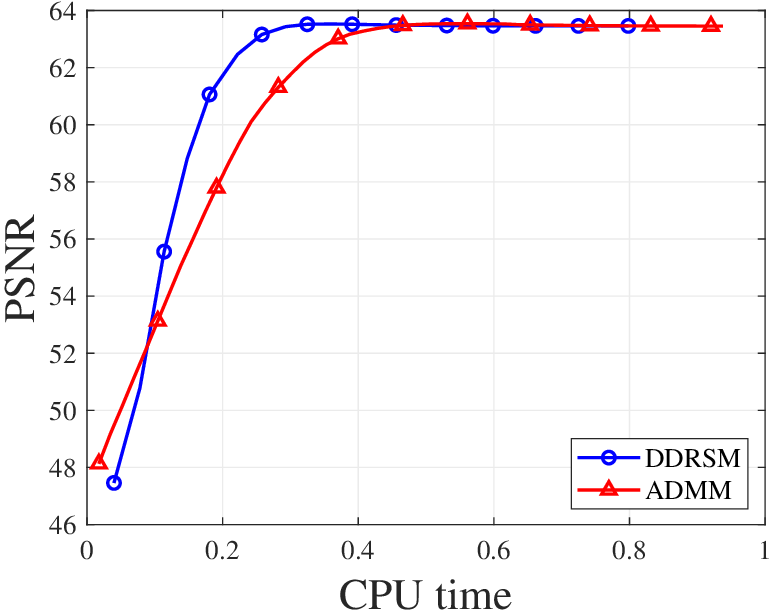}
		\caption{$m=3000, n=1000, \text{sparsity}=0.02$}
	\end{subfigure}
	
	\begin{subfigure}[b]{0.45\textwidth}
		\centering
		\includegraphics[height=2in]{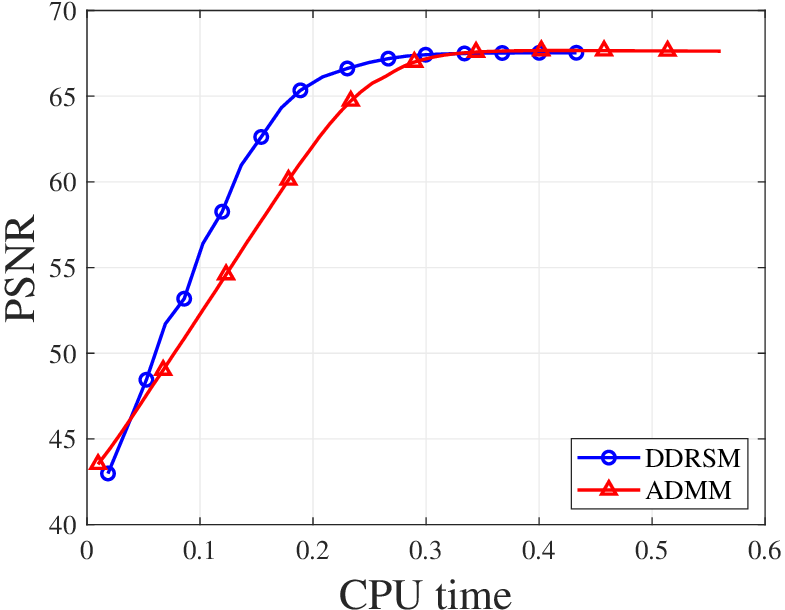}
		\caption{$m=1500, n=1000, \text{sparsity}=0.06$}
	\end{subfigure}
	\hfil
	\begin{subfigure}[b]{0.45\textwidth}
		\centering
		\includegraphics[height=2in]{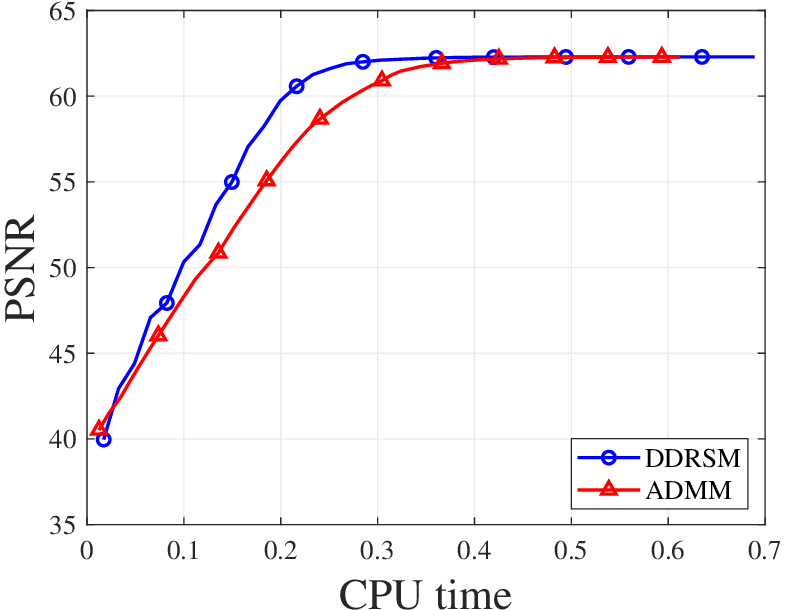}
		\caption{$m=1500, n=1000, \text{sparsity}=0.12$}\label{fig:psnr_d}
	\end{subfigure}
	\caption{PSNR by CPU-time for different sizes of problems}\label{PSNR by CPUtime}
\end{figure}

In summary, the compressed sensing process involves the following steps. Firstly, we generate a sparse signal $x$ with components scaled to the range $[0,1]$. We then transform the signal using the discrete Fourier transform (DFT). The transformed signal is multiplied by the measurement matrix $M$ and perturbed by adding Gaussian random noise with a variance of $0.01$, resulting in the observed data $v$. Secondly, the sparse signal is reconstructed from $M$ and $v$ using the DDRSM method with the nonconvex regularizer ($q=\frac{1}{2}$). This completes the compressed sensing process.

For our experiments, we followed a specific procedure. We generated a sparse vector $x^G \in \mathbb{R}^n$, representing the original signal we aim to recover. We applied the inverse Fourier transform to $ x^G $ to obtain the signal in the time domain. Gaussian random noise was intentionally added to the measurements to introduce a certain level of distortion, resulting in the observed data $ v $. We utilized the peak signal-to-noise ratio (PSNR) metric to assess the quality of the reconstructed signals by comparing the reconstructed signal $ x^* $ obtained from the algorithm with the ground truth $ x^G $. The PSNR values were calculated using the MATLAB function \texttt{psnr}, which compares the reconstructed signals to the original noise-free signals. The PSNR is calculated as follows:
\[
\mathrm{PSNR}=10 \cdot \log _{10}\left(\frac{\text{peakval}^2}{\mathrm{MSE}}\right),
\]
where $ \text{peakval} $ is the maximum possible value of the signal (in this case, set as $ \max_{1 \leq i \leq n} | x_i^G | $), and $ \mathrm{MSE} $ (mean squared error) denotes the average squared difference between corresponding samples of the reconstructed and original signals. That is, for $ x^* $ being the reconstructed signal obtained by the algorithm and $ x^G $ being the ground truth, $ \mathrm{MSE} = \dfrac{1}{n} \| x^* - x^G \|_2^2 $. Higher PSNR values indicate a higher fidelity of the reconstructed signals to the originals.

In the parameter selection process, we adopted a systematic method that involved both a wide-range initial search and a subsequent refined random search to identify optimal hyperparameters for our algorithms. For instance, taking $ \beta $ as an example, we iterated through a range of values (e.g., $ [0.2, 0.6] $ with a step size of $ 0.05 $) while varying $ \delta $ and $ \rho $ within respective ranges. By evaluating the performance metrics (PSNR and CPU time) for each combination, we pinpointed the set of parameters that yielded the highest fidelity in signal reconstruction. We then conducted a refined random search within a smaller range around the optimal values to further fine-tune the hyperparameters. This parameter selection process ensured that our algorithms were optimized to achieve the best possible results within a feasible range of hyperparameters. For every algorithm and every problem size, we searched for the optimal hyperparameters in this manner.

Following our parameter selection methodology, we initialized both ADMM and DDRSM with zero vectors. ADMM terminates when $ \| u^k - u^{k+1} \| \leq \text{tol}_p $ and $ \| \lambda^k - \lambda^{k+1} \| \leq \text{tol}_d $. DDRSM adds an additional termination condition: $ \| E_\beta(x^k, \bar{\lambda}^k) \| \leq \text{tol}_E $, allowing for termination under any of these criteria. To determine the tolerance values, we conducted several runs on small-scale problems up to the maximum iteration number to ensure convergence. We then estimated the tolerance at convergence for both algorithms and derived an empirical relationship between this tolerance and the problem size, which guided our selection of tolerances for the numerical experiments.

Figure \ref{PSNR by CPUtime} displays the obtained PSNR values as a function of CPU time, offering insights into the performance of the algorithms. Specifically, we compare the results of DDRSM and ADMM under various problem sizes and sparsity levels using the LASSO model with \( \ell_{1/2} \) norm regularization. In the legend, "DDRSM" and "ADMM" denote the respective algorithms.

From Figure \ref{PSNR by CPUtime}, we observe that both DDRSM and ADMM converge to high PSNR levels. DDRSM exhibits a rapid convergence speed, achieving high PSNR values within a shorter time frame compared to ADMM. In contrast, ADMM requires more computational time to reach similar PSNR levels.

Table \ref{tab:1} provides detailed numerical results for the final PSNR, CPU time, and the number of iterations for each algorithm under different problem settings. In the table, the superior results are highlighted in bold, indicating where DDRSM outperforms ADMM.
From the table, it is evident that DDRSM generally achieves higher PSNR values, consumes less CPU time, and requires fewer iterations compared to ADMM. The bolded data clearly demonstrate DDRSM's advantages in terms of faster convergence and better reconstruction quality. Specifically, DDRSM consistently shows improved computational efficiency and solution accuracy over ADMM across various problem instances.

It is noteworthy that in the fourth problem setting, ADMM shows a slight advantage in CPU time over DDRSM, as indicated in Table \ref{tab:1}. However, a closer examination reveals that this does not reflect the overall efficiency of the algorithms. From the corresponding Figure \ref{fig:psnr_d}, we observe that DDRSM attains higher PSNR values more rapidly than ADMM. Specifically, DDRSM's PSNR curve ascends steeply, reaching a higher reconstruction quality in less time. The increased total CPU time for DDRSM is attributed to the prolonged termination phase, during which both algorithms maintain relatively stable PSNR levels without significant improvements. This indicates that while ADMM may terminate slightly earlier, DDRSM is more effective in the critical early stages of convergence, providing faster attainment of superior solutions.

Overall, the experimental results demonstrate that DDRSM is effective in solving sparse recovery problems using the LASSO model with \( \ell_{1/2} \) norm regularization, significantly outperforming ADMM in terms of convergence speed and computational efficiency.

\subsection{Robust alignment by sparse and low-rank decomposition}

The second numerical experiment involves a problem related to Robust Principal Component Analysis (RPCA).
The problem solved by Peng et al.\ \cite{peng2012rasl} is about simultaneously aligning a batch of linearly correlated images despite gross corruption. It is commonly used for batch alignment of face images in face recognition and for automatic alignment of landscape and scene images to facilitate subsequent computer vision processing. If we have two misaligned images represented by $I_1$ and $I_2$, which are defined as functions on a picture space $I_1,I_2:\Omega\subset\mathbb{R}^2\rightarrow[0,255]$, there exists an invertible transformation $\tau: \mathbb{R}^2 \rightarrow \mathbb{R}^2$ such that
$$
I_2(x, y)=\left(I_1 \circ \tau\right)(x, y) := I_1(\tau(x, y)).
$$
Here, we will slightly abuse the notation. 
We use $I(\cdot)$ to represent the image function and $I$ to denote the matrix composed of pixels of the corresponding image. That is, $I(\cdot):\Omega\subset \mathbb{R}^2\rightarrow[0,255]$, while $I\in\mathbb{R}^{w\times h}$, where $w$ and $h$ represent the width and height of the image. In addition, we use $I\circ\tau$ to denote the matrix corresponding to the image function $I(\tau(\cdot))$.

Assuming we have $n$ input images $I_1, I_2, \ldots, I_n$ of the same object, but misaligned with each other, we can find domain transformations $\tau_1, \tau_2, \ldots, \tau_n$ such that the transformed images $I_1 \circ \tau_1, I_2 \circ \tau_2, \ldots, I_n \circ \tau_n$ are well-aligned at the pixel level. This can be equivalently represented as finding a low-rank matrix
$$
D \circ \tau \doteq\left[\operatorname{vec}\left(I_1^0\right),\cdots, \operatorname{vec}\left(I_n^0\right)\right] \in \mathbb{R}^{m \times n},
$$
where $I_j^0 = I_j \circ \tau_j$ for $j = 1, 2, \ldots, n$, $D = \left[\operatorname{vec}\left(I_1\right),\cdots, \operatorname{vec}\left(I_n\right)\right]$, and $\tau$ represents the set of $n$ transformations $\tau_1, \tau_2, \ldots, \tau_n$. Therefore, the problem of batch image alignment can be formulated as the following optimization problem:
$$
\min _{A, \tau} \operatorname{rank}(A) \quad \text{s.t.} \quad D \circ \tau=A.
$$
We introduce a sparse regularization term in the objective function to address misalignment and potential sparse errors. This leads to the formulation of the robust alignment by sparse and low-rank decomposition problem (RASL):
\begin{equation}\label{RASL}
	\min _{A, E, \tau} \operatorname{rank}(A) + \lambda \|E\|_0 \quad \text{s.t.} \quad D \circ \tau=A+E,
\end{equation}
where $  E = \left[ \operatorname{vec}\left( e_1 \right),\ \ldots,\ \operatorname{vec}\left( e_n \right) \right]  $ and $  \lambda > 0  $ is a weighting parameter. Here, $  E  $ is a matrix, and we use $  \| \cdot \|_0  $ to denote the number of nonzero elements in $  E  $, corresponding to the $  \ell_0  $ norm of $  \operatorname{vec}(E)  $.

To handle the computational challenges posed by the nonconvex regularization terms $  \operatorname{rank}(\cdot)  $ and $  \| \cdot \|_0  $, a common approach is to use their convex relaxations, namely the nuclear norm $  \| \cdot \|_*  $ and the $  \ell_1  $ norm $  \| \cdot \|_1  $, respectively. However, to improve model performance and obtain solutions closer to reality, nonconvex regularizers can be introduced, such as the quasi-norms $  \| \cdot \|_{*, 1/2}  $ and $  \| \cdot \|_{1/2}  $. The specific definitions of these norms are as follows:
\begin{equation*}
    \begin{aligned}
        \| A \|_1 &:= \sum_{i=1}^m \sum_{j=1}^n | a_{ij} |, \quad &\| A \|_* &:= \sum_{i=1}^{r} \sigma_i(A), \\[1ex]
        \| A \|_{1/2} &:= \left( \sum_{i=1}^m \sum_{j=1}^n \sqrt{ | a_{ij} | } \right)^2, \quad &\| A \|_{*,1/2} &:= \left( \sum_{i=1}^{r} \sqrt{ \sigma_i(A) } \right)^2,
    \end{aligned}
\end{equation*}
where $  a_{ij}  $ are the elements of $  A  $, $  \sigma_i(A)  $ are the singular values of $  A  $, and $  r = \operatorname{rank}(A)  $.

Based on the relaxation of the regularization terms, we have the convex version of RASL:
\begin{equation}\label{RASLconvex}
    \min_{A, E, \tau} \| A \|_* + \lambda \| E \|_1 \quad \text{s.t.} \quad D \circ \tau = A + E,
\end{equation}
and the nonconvex version of RASL:
\begin{equation}\label{RASLnonconvex}
    \min_{A, E, \tau} \| A \|_{*, 1/2} + \lambda \| E \|_{1/2} \quad \text{s.t.} \quad D \circ \tau = A + E.
\end{equation}
The $  \ell_{1/2}  $ quasi-norms here utilize the quadratic smoothing technique mentioned earlier in equation \eqref{eq:smoothing}, and the subproblems involving these quasi-norms can be solved using equation \eqref{eq:proxq}.

In \cite{peng2012rasl}, an outer iteration addresses the nonlinearity of $  \tau  $. Each $  \tau_i  $ is treated as a nonlinear function with $  p  $ parameters, and $  \tau  $ is differentiable with respect to these parameters. For convenience, we slightly abuse notation and denote $  \tau_i  $ as both the transformation and its parameter vector. Therefore, we have $  \tau = [\tau_1,\ \dots,\ \tau_n ] \in \mathbb{R}^{p \times n}  $.

The outer iteration of the algorithm computes the Jacobian matrix of the objective function with respect to the parameters $  \tau  $, resulting in a linearized problem. The outer iteration is described in the following steps:

\begin{itemize}
    \item \textbf{Step 1:} Compute Jacobian matrices for the transformations:
    \[
    J_i \leftarrow \left. \frac{\partial}{\partial \zeta} \left( \frac{ \operatorname{vec}\left( I_i \circ \zeta \right) }{ \left\| \operatorname{vec}\left( I_i \circ \zeta \right) \right\|_2 } \right) \right|_{ \zeta = \tau_i }, \quad i = 1, 2, \ldots, n;
    \]
    \item \textbf{Step 2:} Warp and normalize the images:
    \[
    D \circ \tau \leftarrow \left[ \frac{ \operatorname{vec}\left( I_1 \circ \tau_1 \right) }{ \left\| \operatorname{vec}\left( I_1 \circ \tau_1 \right) \right\|_2 },\ \ldots,\ \frac{ \operatorname{vec}\left( I_n \circ \tau_n \right) }{ \left\| \operatorname{vec}\left( I_n \circ \tau_n \right) \right\|_2 } \right];
    \]
    \item \textbf{Step 3 (Inner Loop):} Solve problem \eqref{RASLconvex_inner} or \eqref{RASLnonconvex_inner} to obtain $  \left( A^*, E^*, \Delta \tau^* \right)  $;
    \item \textbf{Step 4:} Update transformations: $  \tau \leftarrow \tau + \Delta \tau^*  $.
\end{itemize}

The linearized problem solved in the inner iteration can be viewed as a line search in a gradient descent method. The convex version of the linearized problem is:
\begin{equation}\label{RASLconvex_inner}
    \min_{A, E, \Delta \tau} \| A \|_* + \lambda \| E \|_1 \quad \text{s.t.} \quad D \circ \tau + \sum_{i=1}^n J_i \Delta \tau_i e_i^\top = A + E,
\end{equation}
and the nonconvex version is:
\begin{equation}\label{RASLnonconvex_inner}
    \min_{A, E, \Delta \tau} \| A \|_{*, 1/2} + \lambda \| E \|_{1/2} \quad \text{s.t.} \quad D \circ \tau + \sum_{i=1}^n J_i \Delta \tau_i e_i^\top = A + E,
\end{equation}
where $  J_i  $ is the Jacobian of the $  i  $-th image with respect to the transformation $  \tau_i  $, and $  e_i  $ is the $  i  $-th standard basis vector in $  \mathbb{R}^n  $. In equations \eqref{RASLconvex_inner} and \eqref{RASLnonconvex_inner}, $  D \circ \tau  $ becomes a constant matrix, and $  \Delta \tau  $ is the variable to be solved for.

To solve the inner iterative problem \eqref{RASLconvex_inner}, Peng et al.\ \cite{peng2012rasl} employed the inexact augmented Lagrange multiplier method (iALM), a commonly used algorithm for solving this type of optimization problem. In our work, we apply the DDRSM to solve the inner iterative problem. We compare the two algorithms separately for both convex and nonconvex cases.

\begin{figure}[htbp]
	\centering
	\begin{minipage}{0.4\textwidth}
		\centering
		\includegraphics[width=\textwidth]{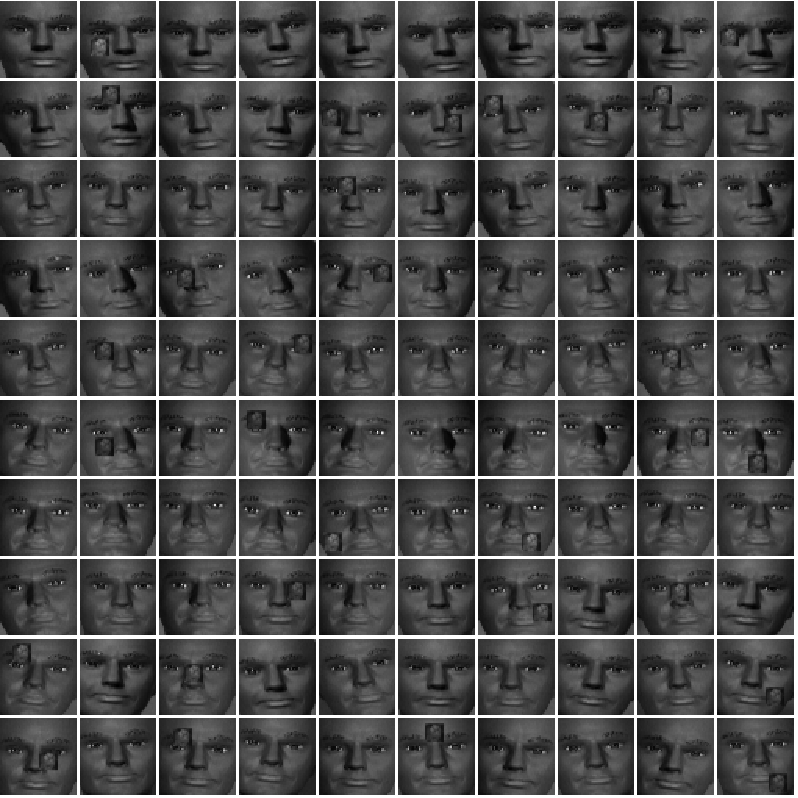}
		\subcaption{dummy faces}
		\label{fig:dummy}
	\end{minipage}
	\hspace{3em}
	\begin{minipage}{0.4\textwidth}
		\centering
		\includegraphics[width=\textwidth]{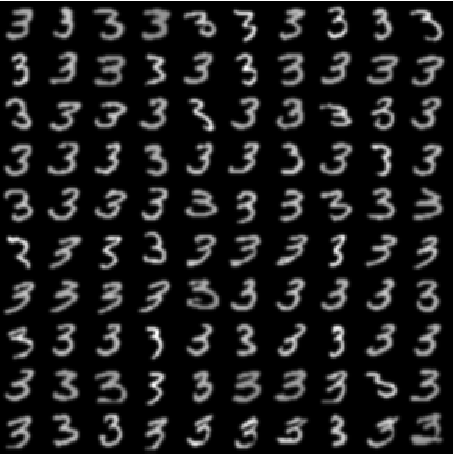}
		\subcaption{digit 3}
		\label{fig:digit}
	\end{minipage}
	\caption{Input images}
	\label{InputImages}
\end{figure}

\vspace{0.1em}

\begin{figure}[htbp]
	\centering
	\begin{minipage}[t]{0.4\textwidth}
		\centering
		\includegraphics[trim=0pt 0pt 0pt 0pt, clip, width=\textwidth]{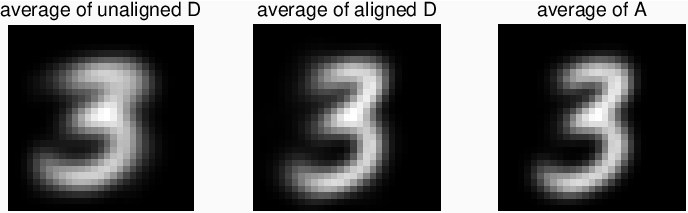}
		\subcaption{convex model by iALM}
	\end{minipage}
	\hfil
	\begin{minipage}[t]{0.4\textwidth}
		\centering
		\includegraphics[trim=0pt 0pt 0pt 0pt, clip, width=\textwidth]{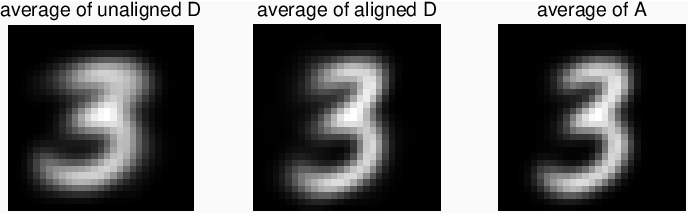}
		\subcaption{convex model by DDRSM}
	\end{minipage}\\
	\vspace{1em}
	\begin{minipage}[t]{0.4\textwidth}
		\centering
		\includegraphics[trim=0pt 0pt 0pt 0pt, clip, width=\textwidth]{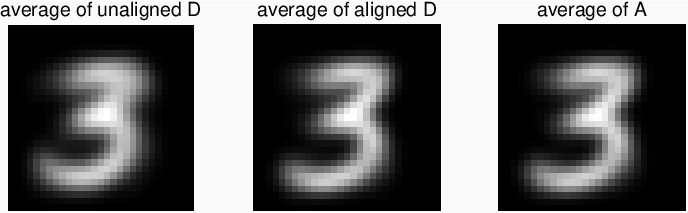}
		\subcaption{nonconvex model by iALM}
	\end{minipage}
	\hfil
	\begin{minipage}[t]{0.4\textwidth}
		\centering
		\includegraphics[trim=0pt 0pt 0pt 0pt, clip, width=\textwidth]{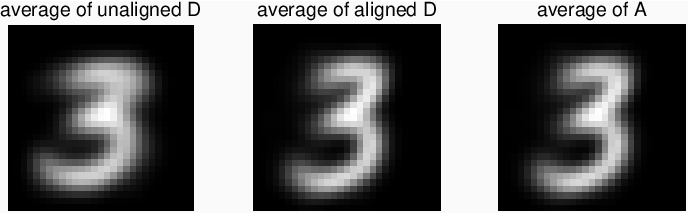}
		\subcaption{nonconvex model by DDRSM}
	\end{minipage}\\
	\vspace{1em}
	\caption{Digit 3 averaged image for both models by both algorithms}
	\label{averaged images}
\end{figure}

\begin{figure}[htbp]
	\centering
	\begin{minipage}[t]{0.3\linewidth}
		\centering
		\setlength\tabcolsep{3pt}
		\includegraphics[width=\linewidth]{ialm_c_digit_1.eps}
		\subcaption{Input images}
	\end{minipage}
	\hfill
	\begin{minipage}[t]{0.3\linewidth}
		\centering
		\setlength\tabcolsep{3pt}
		\includegraphics[width=\linewidth]{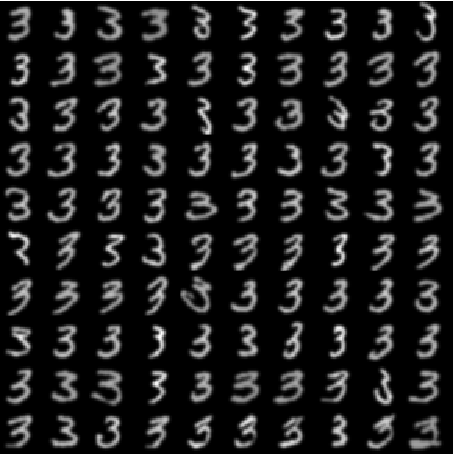}
		\subcaption{aligned images by iALM}
	\end{minipage}
	\hfill
	\begin{minipage}[t]{0.3\linewidth}
		\centering
		\setlength\tabcolsep{3pt}
		\includegraphics[width=\linewidth]{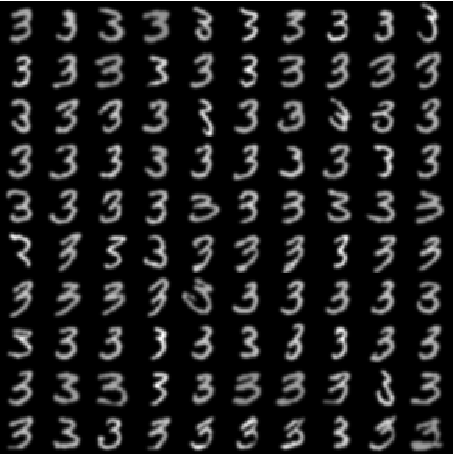}
		\subcaption{aligned images by DDRSM}
	\end{minipage}
	\centering
	\caption{Aligned digit 3 images under convex model}\label{aligned images digit 3 convex}
\end{figure}

\begin{figure}[htbp]
	\centering
	\begin{minipage}[t]{0.3\linewidth}
		\centering
		\setlength\tabcolsep{3pt}
		\includegraphics[width=\linewidth]{ialm_c_digit_1.eps}
		\subcaption{Input images}
	\end{minipage}
	\hfill
	\begin{minipage}[t]{0.3\linewidth}
		\centering
		\setlength\tabcolsep{3pt}
		\includegraphics[width=\linewidth]{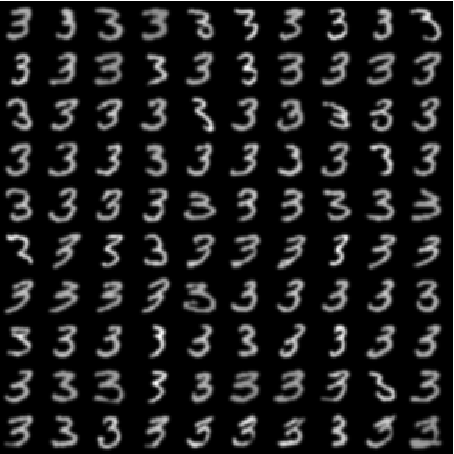}
		\subcaption{aligned images by iALM}
	\end{minipage}
	\hfill
	\begin{minipage}[t]{0.3\linewidth}
		\centering
		\setlength\tabcolsep{3pt}
		\includegraphics[width=\linewidth]{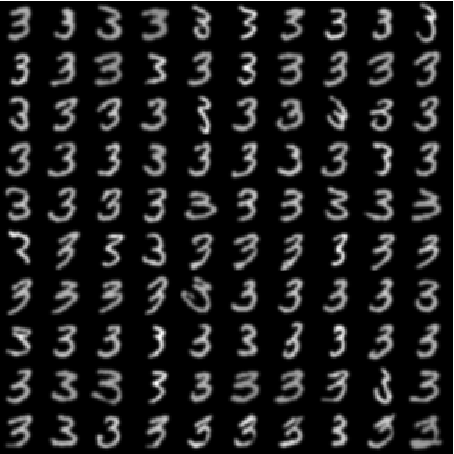}
		\subcaption{aligned images by DDRSM}
	\end{minipage}
	\centering
	\caption{Aligned digit 3 images under nonconvex model}\label{aligned images digit 3}
\end{figure}

\begin{figure}
	\centering
	\begin{minipage}{0.45\textwidth}
		\centering
		\setlength\tabcolsep{3pt}
		\includegraphics[width=\linewidth]{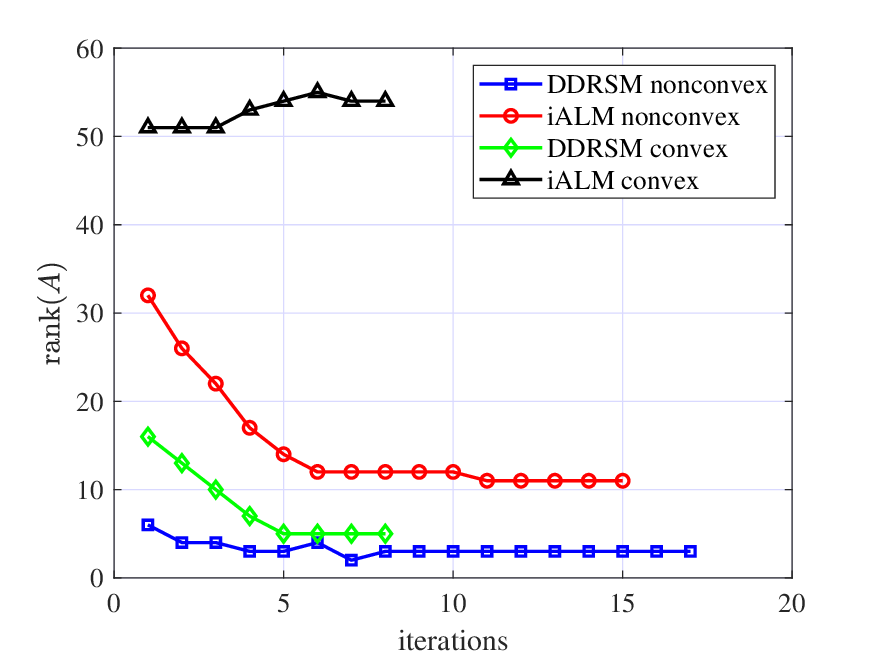}
		\subcaption{rank of $A$ by outer iterations}\label{dummy face chart rank}
	\end{minipage}
	\hfill
	\begin{minipage}{0.45\textwidth}
		\centering
		\setlength\tabcolsep{3pt}
		\includegraphics[width=\linewidth]{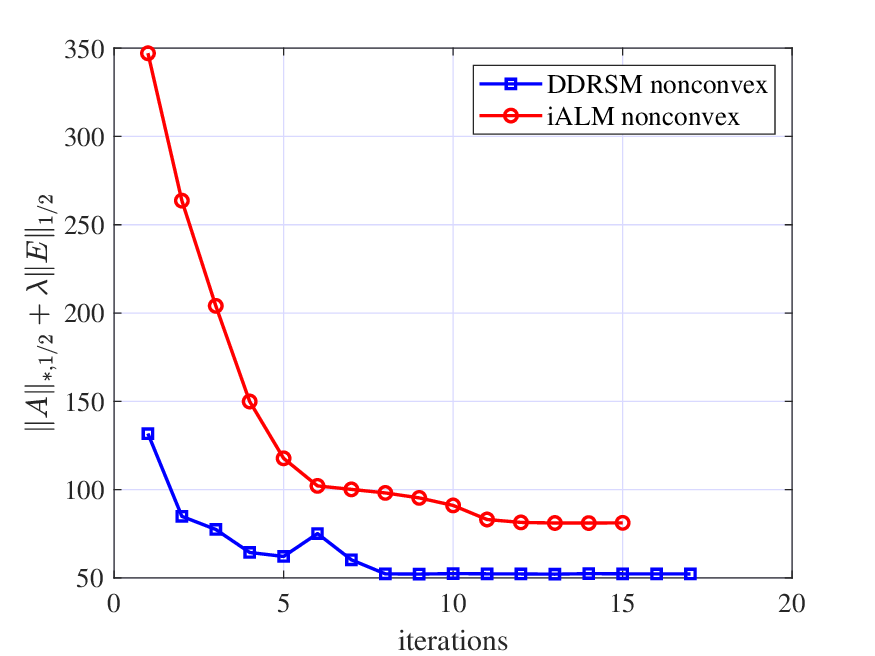}
		\subcaption{objective function for nonconvex case by outer iterations}\label{dummy face chart obj}
	\end{minipage}
	\caption{Charts for dummy face example}\label{dummy face chart}
\end{figure}

Our numerical experiments used two examples: the handwritten digit ``3'' and the dummy face, as shown in Figure~\ref{InputImages}. The input images for these examples were not preprocessed. The example with the digit ``3'' was mainly used to demonstrate the intuitive performance of the algorithm. The average images obtained by the algorithms are shown in Figure~\ref{averaged images}. Under the convex model, both algorithms yield satisfactory results. However, in the nonconvex case, as shown in Figure~\ref{aligned images digit 3}, the iALM algorithm faces difficulties in solving the nonconvex model, as it cannot align the input images effectively. In contrast, the DDRSM algorithm successfully achieves the desired outcome.

Turning our attention to Figure~\ref{dummy face chart}, it includes two subplots presenting a comparative analysis of CPU times for the DDRSM and iALM algorithms in both convex and nonconvex scenarios. The legend in the figure denotes ``DDRSM nonconvex'' for DDRSM on the nonconvex model, ``DDRSM convex'' for DDRSM on the convex model, ``iALM nonconvex'' for iALM on the nonconvex model, and ``iALM convex'' for iALM on the convex model.

As shown in Figures~\ref{aligned images dummy face 1} and \ref{aligned images dummy face 2}, both algorithms correctly align the images for convex and nonconvex models in the dummy face example. However, Figure~\ref{dummy face chart} reveals that the solution obtained for the nonconvex model has a lower overall rank for matrix $A$ than the convex model. Additionally, the DDRSM algorithm consistently achieves a low-rank matrix $A$ for both convex and nonconvex models. 
It is particularly effective in obtaining a low-rank $A$ in the early iterations for the nonconvex model.

For a more detailed comparison in the nonconvex scenario, Figure~\ref{dummy face chart obj} demonstrates that the DDRSM algorithm achieves a faster decrease in the objective function than iALM. As a result, it obtains a solution with a smaller objective function value.

In conclusion, these numerical experiments suggest that DDRSM not only effectively solves the robust alignment problem in both convex and nonconvex models but also outperforms iALM, especially in the nonconvex case. DDRSM demonstrates superior convergence speed and solution quality, highlighting its potential for practical applications in image alignment tasks involving significant corruption and misalignment.

\begin{figure}[htbp]
	\centering
	\begin{minipage}{0.45\textwidth}
		\centering
		\setlength\tabcolsep{3pt}
		\includegraphics[width=\linewidth]{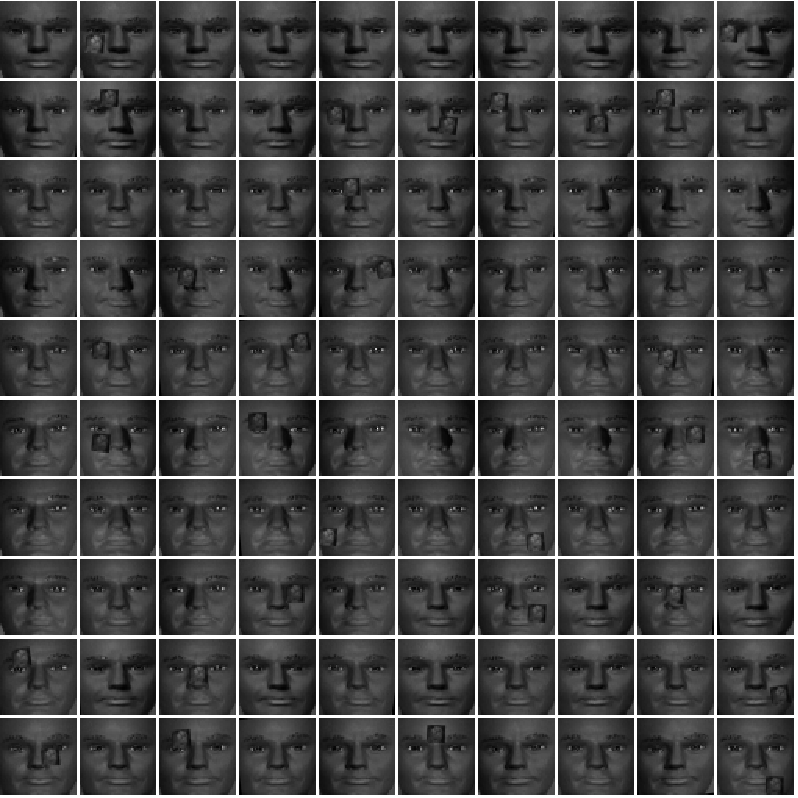}
		\subcaption{aligned images by iALM}
	\end{minipage}
	\hfill
	\begin{minipage}{0.45\textwidth}
		\centering
		\setlength\tabcolsep{3pt}
		\includegraphics[width=\linewidth]{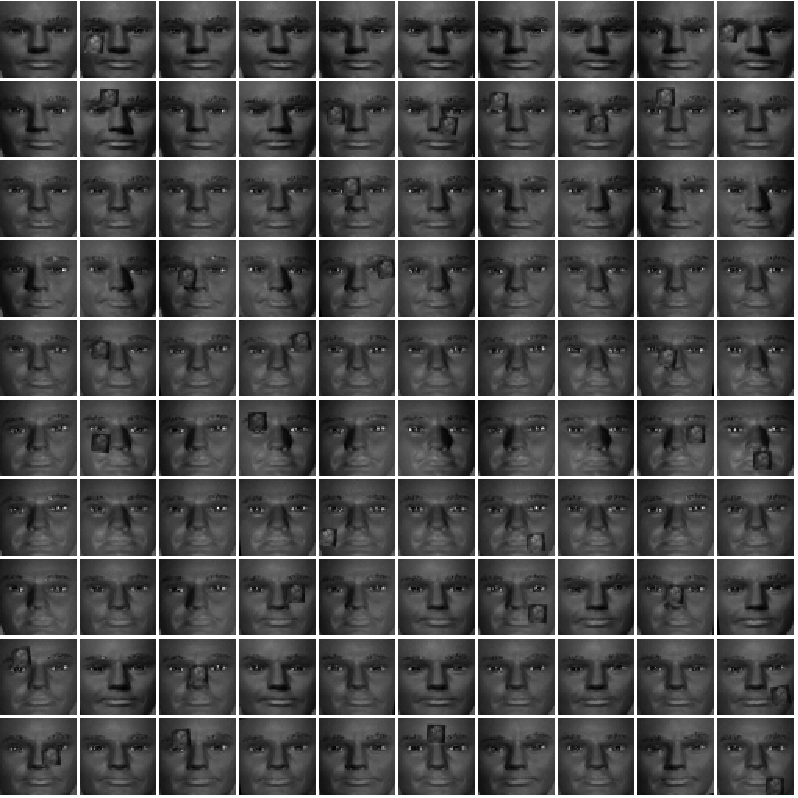}
		\subcaption{aligned images by DDRSM}
	\end{minipage}
	\centering
	\caption{Aligned images under convex model}\label{aligned images dummy face 1}
\end{figure}

\begin{figure}[htbp]
	\centering
	\begin{minipage}{0.45\textwidth}
		\centering
		\setlength\tabcolsep{3pt}
		\includegraphics[width=\linewidth]{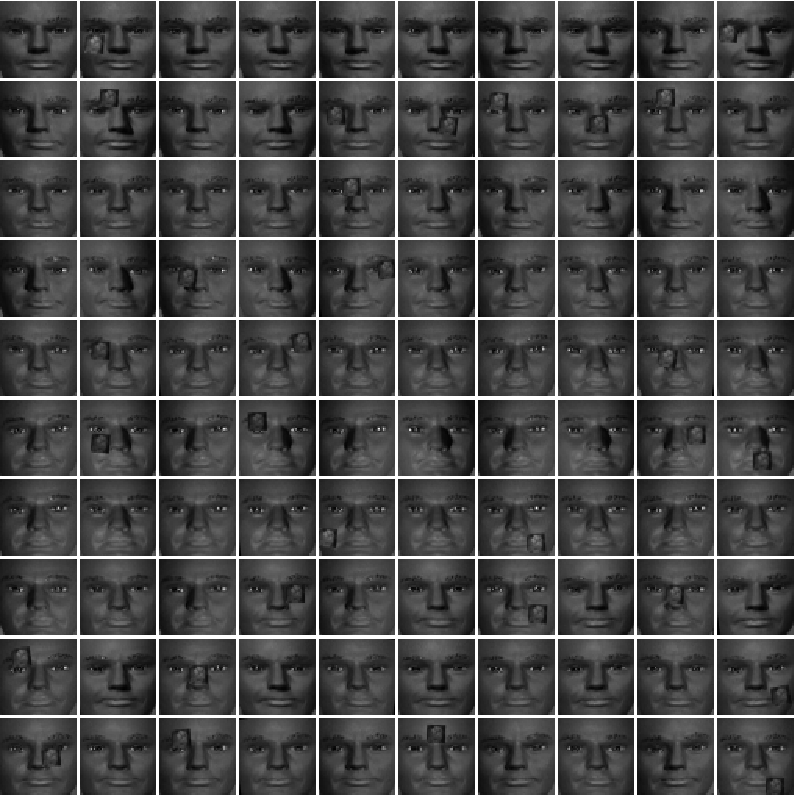}
		\subcaption{aligned images by iALM}
	\end{minipage}
	\hfill
	\begin{minipage}{0.45\textwidth}
		\centering
		\setlength\tabcolsep{3pt}
		\includegraphics[width=\linewidth]{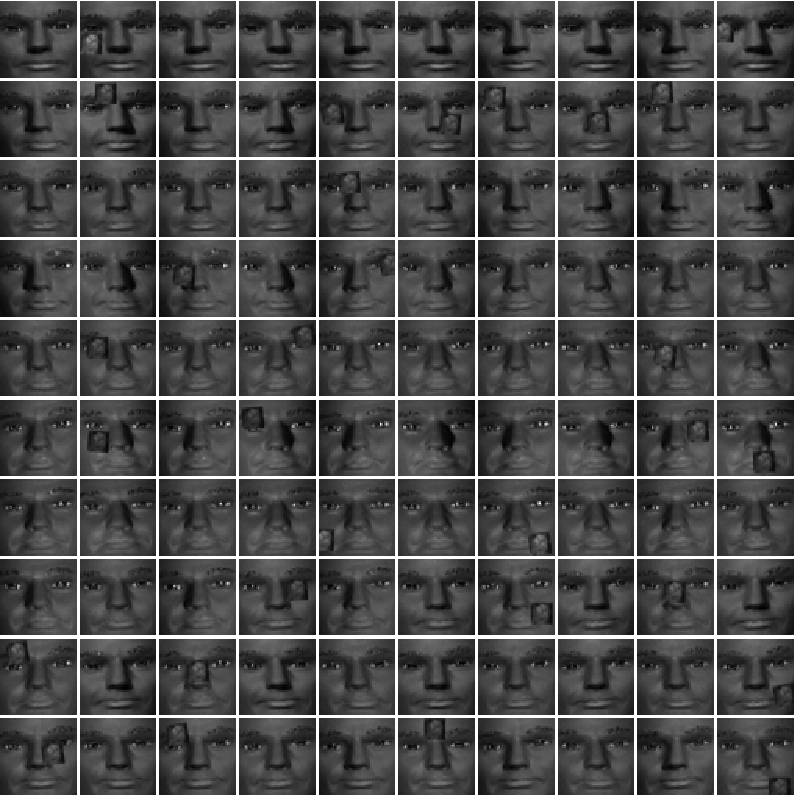}
		\subcaption{aligned images by DDRSM}
	\end{minipage}
	\centering
	\caption{Aligned images under nonconvex model}\label{aligned images dummy face 2}
\end{figure}

\begin{figure}[htbp]
	\centering
	\begin{minipage}{0.45\textwidth}
		\centering
		\setlength\tabcolsep{3pt}
		\includegraphics[width=\linewidth]{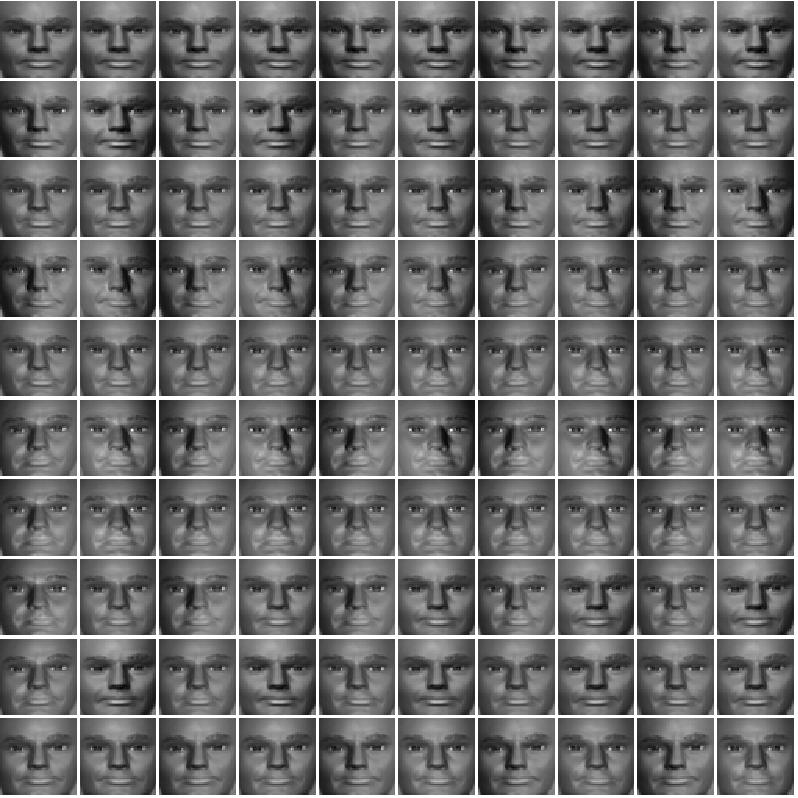}
		\subcaption{adjusted images by iALM}
	\end{minipage}
	\hfill
	\begin{minipage}{0.45\textwidth}
		\centering
		\setlength\tabcolsep{3pt}
		\includegraphics[width=\linewidth]{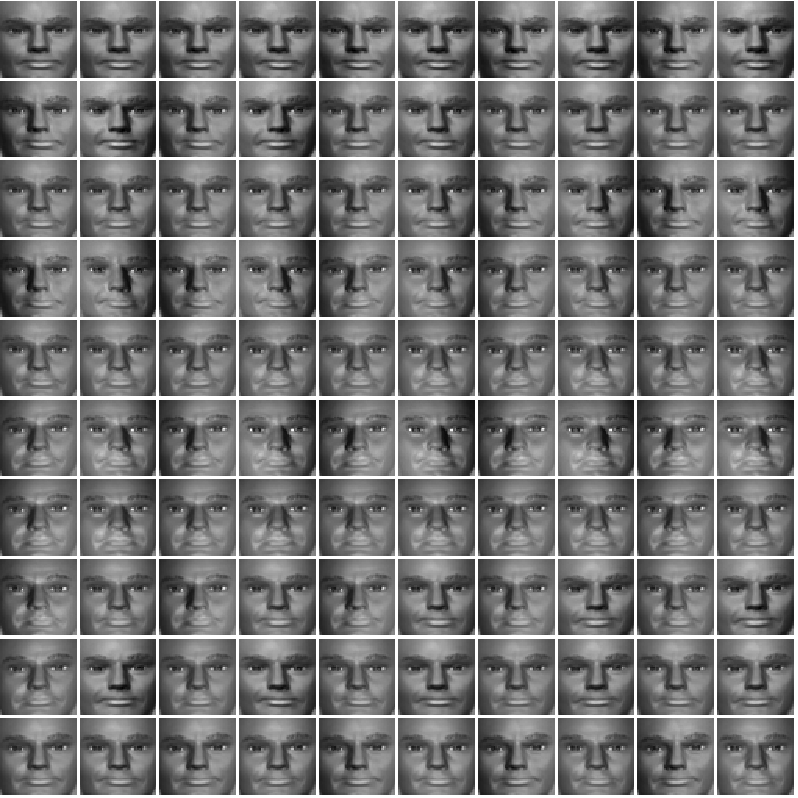}
		\subcaption{adjusted images by DDRSM}
	\end{minipage}
	\centering
	\caption{Aligned images adjusted by sparse error under convex model}\label{aligned images adjusted by sparse error dummy face convex}
\end{figure}

\begin{figure}[htbp]
	\centering
	\begin{minipage}{0.45\textwidth}
		\centering
		\setlength\tabcolsep{3pt}
		\includegraphics[width=\linewidth]{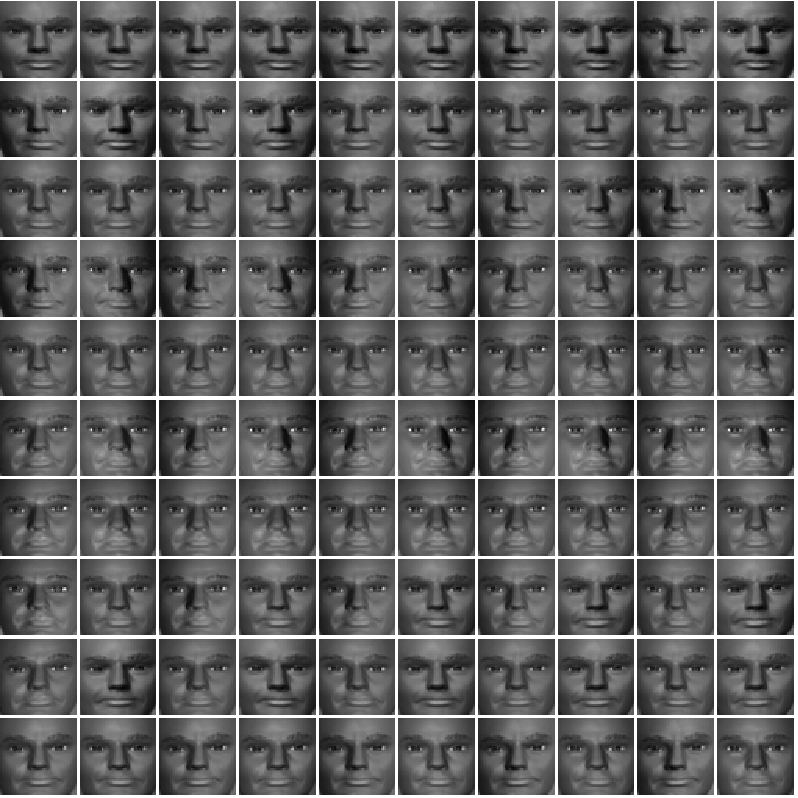}
		\subcaption{aligned images by iALM}
	\end{minipage}
	\hfill
	\begin{minipage}{0.45\textwidth}
		\centering
		\setlength\tabcolsep{3pt}
		\includegraphics[width=\linewidth]{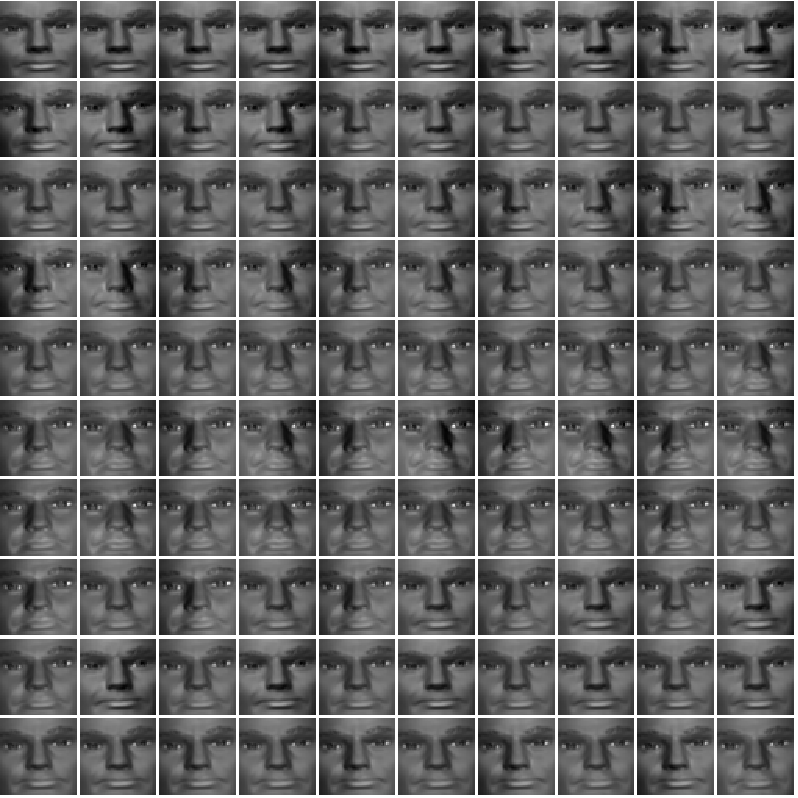}
		\subcaption{aligned images by DDRSM}
	\end{minipage}
	\caption{Aligned images adjusted by sparse error under nonconvex model}\label{aligned images adjusted by sparse error dummy face nonconvex}
\end{figure}

\section{Conclusion}

In this paper, we presented a convergence analysis of the distributed Douglas-Rachford splitting method (DDRSM) for weakly convex problems. The proposed approach incorporates an error bound scheme for optimization methods and leverages variational inequality theory. 
We have shown that DDRSM is adaptable to both convex and nonconvex problems with minor adjustments while maintaining convergence properties.
To validate our theoretical findings, we implemented the algorithm for the LASSO model in signal denoising and for robust alignment using sparse and low-rank decomposition. The numerical results demonstrate the effectiveness of the algorithm in solving weakly convex versions of regularization problems involving nonconvex norms, such as the nuclear quasi-norm and the $ \ell_{1/2} $ quasi-norm.

\bibliography{DDRSM-nonconvex}

\end{document}